\documentclass[journal, 11 pt]{IEEEtran}
\usepackage[utf8]{inputenc}

\usepackage{graphicx}
\usepackage{amsmath,amsfonts,amsthm}								
\usepackage{amssymb}
\usepackage{mathtools}	
\usepackage{mathabx}
\usepackage{midfloat}
 \usepackage{stfloats}
\usepackage{algorithm}
\newcommand{\numberset}{\mathbb}							

\newcommand{\R}{\numberset{R}}

\usepackage{amscd}										
\usepackage{booktabs}										
\usepackage{graphics} 										
\usepackage{epsfig} 										
\usepackage{contour}										
\usepackage{bm}
\usepackage{times} 											
\usepackage{rotating}										
\usepackage{cite} 									
\usepackage{listings}										
\usepackage{xcolor}

\usepackage{url}
\usepackage{array}

\newcommand{\blue}[1]{{\textcolor{black}{#1}}}

\DeclareMathAlphabet{\mathpzc}{OT1}{pzc}{m}{it}
\DeclareMathAlphabet{\mathcal}{OMS}{cmsy}{m}{n}

\newtheorem{lemma}{Lemma}
\newtheorem{corollary}{Corollary}                   						
\newtheorem{remark}{Remark}
\newtheorem{theorem}{Theorem}

\usepackage{cases}
\usepackage{tikz}				
\usetikzlibrary{arrows,shapes,backgrounds,calc,positioning,patterns}
\usepackage{balance}

\newcommand{\vect}[1]{\boldsymbol{#1}}

\DeclareMathOperator{\rank}{rank}
\DeclareMathOperator{\im}{Image}
	
\DeclareMathOperator{\Dom}{Dom}	
\DeclareMathOperator{\col}{col}

\newcommand{\setof}[1]{\left\{#1\right\}} 

\newcommand{\brackets}[1]{\left[#1\right]}
\newcommand{\braces}[1]{\left(#1\right)}
\newcommand{\angles}[1]{\langle#1\rangle}
\newcommand{\norm}[1]{\|#1\|}

\usepackage{pifont}

\usepackage{nicefrac}

\title{Novel current actuated piezoelectric composite model with fully dynamic electromagnetic field}
\author{Matthijs C. de Jong and Jacquelien M. A. Scherpen 
		\thanks{ M.C. de Jong and J.M.A. Scherpen are with Jan C. Wilems Center for Systems and Control, ENTEG, Faculty of Science and Engineering, University of Groningen, Nijenborgh 4, 9747 AG Groningen, the Netherlands (email: \{  m.c.dejong, j.m.a.scherpen\}@rug.nl).}
		\thanks{We would like to acknowledge Kirsten A. Morris for her insights and suggestions during the development of this work.}
	}

\begin{document}

\maketitle
\begin{abstract}
  
    In this paper we propose a novel current actuated fully dynamic piezoelectric composite model. This new model complements the existing modeling framework of deriving piezoelectric models. We show that the novel composite model is well-posed. Furthermore, we consider two approximation methods, FEM and MFEM and investigate the stabilizability properties. Finally, we include the closed-loop behaviour in the numerical results.
    
\end{abstract}


\section{Introduction}
        
         \blue{\IEEEPARstart{A}{piezoelectric}  actuator is a piece of piezoelectric material, which is sandwiched between two layers of electrodes. By an electric stimulus, such as voltage, charge, or current, it can compress or elongate in one or more directions. By gluing this type of piezoelectric actuator onto the surface of a mechanical substrate, the deformation of the actuator incurs shear stress in the substrate, which curves the composition. A mechanical substrate with one or more piezoelectric actuators we refer to as a piezoelectric composite, see Fig \ref{fig:piezoelectriccomposite} and is useful in high precision applications. An other specific type of electric actuator is the piezoelectric beam, where an electric stimulus acting on the transverse axis only incurs deformation on the longitudinal direction.}\\
         
	From a control perspective, there exist roughly two rubrics of applications for piezoelectric composites, i.e. vibration control and shape control. The former has applications in acoustic devices \cite{Ralib2015} or suppression of vibrations in mechanical applications \cite{Samikkannu2002vibrationdamp}. The latter rubric envelopes applications such as flexible wings \cite{CHUNG2009136}, inflatable space structures \cite{voss2010port} and deformable mirrors \cite{Radzewicz04}.  Often, in applications including inflatable space structures and deformable mirrors, one side of the substrate has a specific function (e.g. reflecting electromagnetic waves). Therefore, we consider in this work a composite where the piezoelectric actuator is attached to one side of the purely mechanical layer, see Fig \ref{fig:piezoelectriccomposite} for a depiction.\\
	
	\begin{figure}[!t]
		\centering
		\label{fig:piezoelectriccomposite}
		\includegraphics[width=1\columnwidth]{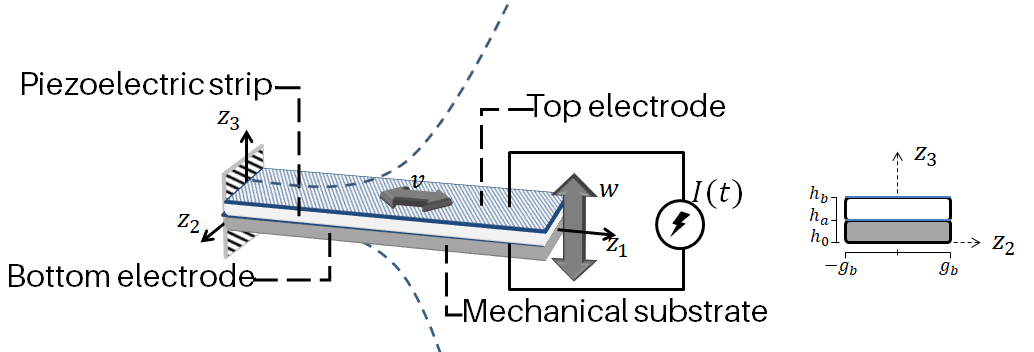}
		\caption{Piezoelectric composite with longitudinal deflection $v$ and transverse deflection $w$.}
	\end{figure}
	
	The dynamics describing the behaviour of a piezoelectric composite originate from continuum mechanics (mechanical domain) and Maxwell's equations (electromagnetic domain) and results in a set of partial differential equations (PDE). PDEs with different structure and properties are derived by changing the assumptions for either the mechanical or electromagnetic domain. For the mechanical domain, a different type of beam theory can be assumed, and non-linear phenomena can be incorporated. For the electromagnetic domain, different assumptions on the treatment of Maxwell's equations result in different dynamics and coupling of the electric, magnetic, and mechanical quantities. Often, in literature, the assumptions for the electromagnetic domain are typified as a \textit{static electric-field}, a \textit{quasi-static electric-field}, or a \textit{fully dynamic electromagnetic field}. Recent efforts show that the treatment of the electromagnetic domain and choice of input severely alters the controllability and stabilizability properties of the dynamics, see for instance \cite{MenOCDC2013,MenOACC2014,MenOMTNS2014,MenOSIAM2014,voss2011CDC,VenSMMS2011,VenSSIAM2014,voss2010port,OzerMCSS2015,OzerTAC2019}.\\
	
	The traditional choice of actuation is voltage actuation. 
    Voltage actuated linear infinite-dimensional piezoelectric beam models are exactly controllable and exponentially stabilizable in the case of a static or quasi-static electric-field, see \cite{komornik2005}, \cite{KapitonovMM07}, and \cite{LASIECKA2009167}. In \cite{MenOSIAM2014}, the fully dynamic voltage actuated beam model is shown to be asymptotically stabilizable, for almost all system coefficients and exponentially stabilizable for a small set of system coefficients. In \cite{OzerMCSS2015} it has been shown that the fully dynamical beam is polynomial stabilizable when certain conditions on the physical coefficients are satisfied.\\
	
	Due to less-hysteretic behaviour of charge and current actuated piezoelectric systems, recent studies involved the stabilizability of such models. The charge actuated models show similar stabilizability properties to voltage actuated models \cite{MenOSIAM2014}. There are two ways to derive current actuated piezoelectric models. The simplest is obtained by adding a dynamical equation on the boundary to convert the charge input into a current input, which we will refer to as \textit{current-through-the-boundary} actuation. Physically this corresponds to incorporating some electric circuitry. The other way of obtaining current actuated systems is by careful considerations on the electromagnetic domain, resulting in a \textit{purely} current actuated system. The modelling of purely current actuated piezoelectric systems is described in \cite{MenOSIAM2014,OzerTAC2019}.\\

	The stabilizability of current actuated piezoelectric beams has been investigated under a fully dynamic electromagnetic field, where the model is derived using magnetic vector potentials that require an additional gauge condition to guarantee a solution in \cite{MenOCDC2014}. It has been shown that the control input is bounded in the energy space and can be asymptotically stabilized. Recently, in \cite{OzerTAC2019}, the current actuated piezoelectric beam model from \cite{MenOCDC2014} has been interconnected with a substrate. Due to the use of magnetic vector potentials, the model allows for both current and charge input. It has been shown that the reduced electrostatic models with charge and current-through-the-boundary models are respectively exponentially and asymptotically stabilizable and the derived purely current actuated fully dynamical model is not stabilizable.\\
	
	In \cite{VenSSIAM2014} two current actuated non-linear piezoelectric composite systems are derived with a port-Hamiltonian \cite{port-HamiltonianIntroductory14} perspective. One system is derived with a purely current actuated system with quasi-static electric-field assumption and the other a current-through-the-boundary system with the fully dynamic electromagnetic field assumption. The approximations of the fully dynamic piezoelectric beam \cite{VenSSIAM2014}, using the mixed finite element method \cite{GoloSchaft2004, VenSMMS2011}, has been shown in \cite{voss2011CDC} to satisfy a necessary condition for asymptotic stabilizability. In contrast, the quasi-static system do not satisfy this condition and thus are not stabilizable. \\
	
	
    In this work, we present a novel fully dynamic piezoelectric composite model with purely current input. We use, to the best of our knowledge, a new and elegant treatment of the electromagnetic domain that does not require the use of a gauge function. \blue{The new fully dynamic electromagnetic model is a more comprehensive version, in terms of the electromagnetic treatment, of the quasi-static electromagnetic model in \cite{voss2011CDC,VenSSIAM2014}, which has been shown to be not stabilizable}. The new model is derived by using a novel treatment of Maxwell's equations producing the electromagnetic dynamics in terms of the magnetic-flux. Besides the well-posedness, we also investigate the stabilizability of this new model and supplement these results with some numerical results.\\
	
	
	
	The models presented here are derived under the assumption of linear relationships, and for the mechanical domain, we consider the Euler-Bernoulli beam theory. In the next section, we treat the derivation of the novel model and compare the treatment of Maxwell's equations to the existing methods. 
     In Section III we show that the novel model is well-posed. In Section IV, the novel purely current actuated piezoelectric model is discretized using two approximations methods, i.e. the finite element method (FEM) and the mixed finite element method (MFEM). In Section V, we investigate stabilizability properties of the approximated composites and in Section IV some illustrative simulation are presented to accompany the stabilizability results.

\section{Derivation physical model}
        In this section, the derivation of the novel current actuated piezoelectric composite model is done and at the end of this section we discuss the differences with respect to existing current actuated models. The piezoelectric composite investigated in this work is a piezoelectric actuator superimposed and fixed on top of a purely mechanical substrate. See Fig \ref{fig:piezoelectriccomposite} for a depiction of the system. The governing equations are derived by modelling the piezoelectric actuator and interconnecting it with the governing equations for the mechanical substrate. We focus on current as electric stimulus, which allows the composite to deform in the longitudinal $z_1$ and transverse direction $z_3$ as a result of shear stresses between the piezoelectric actuator and mechanical substrate.\\
        
        The dynamical equations governing piezoelectric beams originate from Maxwell's equations \cite{eom2013maxwell} and continuum mechanics \cite{carrera2011beam}. The coupling between the electromagnetic and mechanical characteristics, such as the electric displacement $\vect{D}$, electric-field $\vect{E}$ and mechanical stress $\vect{\sigma}$ and strain $\vect{\epsilon}$, are described using the piezoelectric constitutive relations
        	\begin{alignat}{1}\label{eq:Constitutive_relations}
    	\left[\begin{array}{c}
            	\vect{\sigma}\\
            	\vect{D}
        	\end{array}\right]= & \left[\begin{array}{cc}
            	C & -e\\
            	e & \varepsilon
            	\end{array}\right]\left[\begin{array}{c}
            	\vect{\epsilon}\\
            	\vect{E}
        	\end{array}\right],
    	\end{alignat}
        where $C$ is the stiffness matrix, $e$ denotes the piezoelectric constants matrix, and $\varepsilon$ the diagonal permittivity matrix. \blue{For the one-dimensional piezoelectric actuator, where the electric stimuli is applied along $E_3$, the components $E_1=E_2=0$. }The constitutive relations are reduced to scalar equations. Let $C_{11}$, $C_{55}$, $\beta:=\varepsilon_{33}$ and $\gamma:=e_{15}$ represent the mechanical stiffness, shear constant,  magnetic permittivity, and piezoelectric constant, respectively. Then, the one-dimensional piezoelectric constitutive scalar relations are given by 
        \blue{
    	\begin{alignat}{1}\label{eq:Constitutive_relations_scalar}
    	\left[\begin{array}{c}
            	\sigma_{11}\\
            	\sigma_{13}\\
            	D_3
        	\end{array}\right]= & \left[\begin{array}{ccc}
            	C_{11} & 0&-\gamma\\
            	0 & {2}C_{55}& 0\\
            	\gamma & 0&\nicefrac{1}{\beta}
            	\end{array}\right]\left[\begin{array}{c}
            	\epsilon_{11}\\\epsilon_{13}\\
            	E_3
        	\end{array}\right],
    	\end{alignat}
    	which takes care of the coupling between the mechanical and electromagnetic part. }
        
        For the mechanical part let $v(z_1,t)$ denote the longitudinal displacement with respect to the normal of the cross-section along the $z_1-$axis and let $w(z_3,t)$ denote the transversal displacement along the $z_3-$axis. Then, the displacement vector $\vect{u}$ for an Euler-Bernoulli beam \cite{carrera2011beam} can be described by 
    	\begin{align}\label{eq:devormation_vvect}
    	\vect{u}= & [\begin{array}{ccc}
        	    v(z_1,t)-\vect{z}_3\tfrac{\partial}{\partial z_1}w(z_1,t), & 0, & w(z_1,t)]^{T}
        	\end{array}.
    	\end{align}
    	where $\vect{z}_3$ denotes the unit-vector in the $z_3$-direction.
       
    \blue{
        The strain is derived form \eqref{eq:devormation_vvect} is as follows,
        \begin{align}\label{eq:strainequations}
            \begin{split}
                    \epsilon_{11}&=\frac{\partial}{\partial z_1}\left( v(z_1)-\vect{z}_3\frac{\partial}{\partial z_1}w(z_1) \right),\\
                    \epsilon_{13}&=0,
                \end{split}
        \end{align}
        where $\epsilon_{13}$ represents the shear-strain, which is not included in the Euler-Bernoulli beam theory \cite{carrera2011beam}. Therefore, the strain, stress, and electric displacement of interest are given by the first row of \eqref{eq:strainequations}, the first row of \eqref{eq:Constitutive_relations_scalar}, and the the third row of \eqref{eq:Constitutive_relations_scalar}, respectively. }
        
    
        Let $T$ and $V$ denote the kinetic and potential energies, $W$ denotes the work applied to the system using a current, and let the subscripts $e$ and $m$ refer to the electrical and mechanical domains, respectively, then the Lagrangian    
        \begin{align}\label{eq:Lagrangian_standard}
            \mathcal{L}:=T_{m}+T_{e}-(V_{m}+V_{e})+W
        \end{align}
         can be used in the derivation of the set of dynamical equations for the piezoelectric beam. To use \eqref{eq:Lagrangian_standard} we require to treat Maxwell's equations to provide us with the necessary components.
    	%
        Therefore, let the vectors $\vect{B}$, $\vect{H}$, and $\vect{J}_i$ represent respectively the magnetic field, magnetic field intensity, and impressed current density. Furthermore, let the magnetic permeability of the material by $\mu$, and the charge density by $\sigma_v$. Then, Maxwell's equations  \cite{eom2013maxwell} are described by the four laws;
        	\begin{subequations}
    	\begin{align}
        	\nabla\times \vect{H}&=\vect{J}_i+\frac{\partial \vect D}{\partial t},\label{eq:Ampereslaw} \\
        	\nabla\cdot\vect D&=\sigma_v, \label{eq:GausssElectric} \\
        	\nabla\times \vect{E}&=-\frac{\partial \vect{B}}{\partial t},  \label{eq:Faraday'sLaw} \\
        	\nabla\cdot\vect{B}&=0,  \label{eq:GausssMagnetic}
    	\end{align}
    	\end{subequations}
    	respectively, Ampere's law \eqref{eq:Ampereslaw} that describes the generation of a magnetic field by current densities, Gauss's law \eqref{eq:GausssElectric} of electric-fields, Faraday's law \eqref{eq:Faraday'sLaw} for time varying magnetic fields, and Gauss's law \eqref{eq:GausssMagnetic} of magnetism. Additional to Maxwell's equation we require the following two constitutive relations
    	\begin{align}\label{eq:constrel_Maxwell}
        	\begin{split}
            	\vect{D}&=\beta\vect{E},\\
            	\mu\vect{H}&=\vect{B},
        	\end{split}
    	\end{align}
    	used for material that admits characteristics prescribed by Maxwell's equations. Using \eqref{eq:Ampereslaw}-\eqref{eq:GausssMagnetic} and \eqref{eq:constrel_Maxwell} we can propose a novel way of modelling current actuated piezoelectric systems. 
         The new approach of deriving the dynamics for a current actuated piezoelectric actuator follows a procedure that is inspired by the approach of deriving dynamics for voltage actuated piezoelectric actuators \cite{MenOSIAM2014}. However, in literature, this approach has not been considered for current actuated models.\\

          \blue{We start by defining the magnetic-flux $\Phi$ as follows,
    	\begin{align}\label{eq:magneticflux_flow3}
        	\dot\Phi(z_1):=\int_{0}^{z_1}\dot{B_2}(\xi)d\xi,
    	\end{align}
    	where $B_2$ is the magnetic field component for $z_2$. From \eqref{eq:magneticflux_flow3} the following expressions 
    	\begin{subequations}\label{eq:magneticfluxses}
        	\begin{align}
                	\frac{\partial}{\partial z_1}\dot{\Phi}(z_1)&=\dot{B}_2(z_1),\label{eq:magneticfluxexpression1}\\
                	\frac{\partial}{\partial z_1}{\Phi}(z_1)&={B}_2(z_1),\label{eq:magneticfluxexpression2}
            \end{align}
    	\end{subequations}	
    	are derived, which relates the magnetic-flux and the magnetic field as a function of the spatial variable $z_1$.
    	 Next, we reduce Faraday's law \eqref{eq:Faraday'sLaw} to the scalar equation
    	 \begin{align}\label{eq:Farady_scalar}
    	     \frac{\partial E_3}{\partial z_1}=\frac{\partial B_2}{\partial t},
    	 \end{align}
    	 where we made use of $E_1=E_2=0$. Recall that, the electric-field from the applied current is only generated in the $E_3$ component. Then, from \eqref{eq:Farady_scalar} with the use of \eqref{eq:magneticfluxexpression1} the relation between the electric-field and flux is given by}]
    	 \begin{align}\label{eq:electricfluxandfield}
    	     \dot\Phi(z_1)&=E_3(z_1).
    	 \end{align}
    	 The expression given in \eqref{eq:electricfluxandfield} is used during the derivation of the dynamics using Hamilton's principle \cite{lanczos1970variational}, applied to the Lagrangian \eqref{eq:Lagrangian_standard}. \blue{From now on we omit the spatial dependency of the variables to enhance the readability. Additionally, we denote $\frac{\partial}{\partial z_1}\phi$ as $\phi_{z_1}$, unless otherwise appropriate.}\\
    	
        
        %
      \blue{The energies associated with \eqref{eq:Lagrangian_standard} for the  novel current actuated piezoelectric actuator are presented next. The relation between the magnetic-flux and electric-field \eqref{eq:electricfluxandfield} plays an intrinsic part in the derivation of the dynamical model.}  Let $\rho$ denote the mass density, $V$ the volume of the beam, and let $A,I,I_0$ represent the cross-section and inertias of the beam. Then, the following energies for the piezoelectric beam are given by
      \blue{
      \begin{subequations}\label{eq:FD_energies}
	\begin{align}
			T_{m}&=\frac{1}{2}\int_V \rho(\dot{\vect{u}}\cdot \dot{\vect{u}})~ dV \label{eq:FD_energiesTm} \\ 
			    &=\frac{1}{2}\int_{0}^{\ell}\brackets{\rho\braces{A\dot{v}^2-2I_0\dot{w}_{z_1}\dot{v}+I\dot{w}^2_{z_1}+A\dot{w}^2}}dz_1, \nonumber \\
			V_{m}&=\frac{1}{2}\int_V \vect{\sigma} \cdot \vect{\epsilon}~ dV \nonumber\\ &=\frac{1}{2}\int_V\left[\sigma_{11}\epsilon_{11} \right]~dV \label{eq:FD_energiesVm}\\
			   &=\frac{1}{2}\int_{0}^{\ell}\left[C_{11}\braces{Av_{z_1}^2+Iw_{z_1z_1}^2-2I_0v_{z_1}w_{z_1z_1}} \right. \nonumber\\
			    &\quad\left.-\gamma\dot{\Phi}\braces{Av_{z_1}-I_0w_{z_1z_1}}\right]dz_1,\nonumber\\
			T_{e}&=\frac{1}{2}\int_V \vect{D}\cdot \vect{E}~ dV \nonumber\\ 
			   &=\frac{1}{2}\int_V {D_3} {E_3}~ dV \label{eq:FD_energiesTe}\\ &=\frac{1}{2}\int_{0}^{\ell}\brackets{\frac{1}{\beta}A\dot{\Phi}^2+\gamma\dot{\Phi}\braces{Av_{z_1}-I_0w_{z_1z_1}} }dz_1,\nonumber\\
			V_{e}&=\frac{1}{2}\int_V \vect{H}\cdot \vect{B}~ dV \nonumber\\ 
			&=\frac{1}{2}\int_V \frac{1}{\mu} B_2^2~ dV \label{eq:FD_energiesVe}\\
			    &=\frac{1}{2}\int_{0}^{\ell}\brackets{\frac{1}{\mu}A\Phi_z^2}dz_1,\nonumber
	\end{align}
	\end{subequations}
	where we used \eqref{eq:devormation_vvect} in \eqref{eq:FD_energiesTm}, in \eqref{eq:FD_energiesVm} we made use of \eqref{eq:strainequations}, the first row of \eqref{eq:Constitutive_relations_scalar} and \eqref{eq:electricfluxandfield}, in \eqref{eq:FD_energiesTe} we made use of the third row of \eqref{eq:Constitutive_relations_scalar} and \eqref{eq:electricfluxandfield}, and in \eqref{eq:FD_energiesVe} we made use of \eqref{eq:constrel_Maxwell} and \eqref{eq:magneticfluxexpression2}.} Moreover, the  cross-sections and inertia's in \eqref{eq:FD_energies} are given by
	\begin{align}\label{eq:crossseactionalA}
    	\begin{split}
        	A&:=\int_{h_a}^{h_b}\int_{-g_b}^{g_b}dz_2dz_3=2g_b(h_b-h_a),\\
        	I&:=\int_{h_a}^{h_b}\int_{-g_b}^{g_b}{z}_3^2dz_2dz_3=\frac{2}{3}g_b(h_b^3-h_a^3),\\ 
        	I_0&:=\int_{h_a}^{h_b}\int_{-g_b}^{g_b}\vect{z}_3dz_2dz_3=\left[g_b{z}_3^2\right]_{h_a}^{h_b}=g_b(h_b^2-h_a^2). 
    	\end{split}
	\end{align}
    \blue{
    It is interesting to point out that the coupling between the mechanical and electrical domain in the energies is present in the potential energy of the mechanical part and the kinetic energy of the electromagnetic part. This is an immediate consequence of the altered Lagrangian, by means of a Legendre transformation, where the kinetic and potential energies of the electromagnetic part are switched when compared to the Lagrangian for voltage actuated piezoelectric actuators, see \cite{morrisozer2016modelingcurrent}.
    The definition of the magnetic flux \eqref{eq:magneticfluxses} and relation to the electric-field \eqref{eq:electricfluxandfield} ensure the appropriate behavior of the kinetic and potential energies of the electromagnetic part. Moreover, we circumvent the necessity of using magnetic vector potentials and a gauge function to obtain a well-posed set of equations, as is described in \cite{MenOCDC2014} for a piezoelectric beam and \cite{OzerTAC2019} for piezoelectric composites. 
    }

	The distributed current source $\mathcal{I}(t)$ acts on the surface of the piezoelectric layer where the electrodes (with surface $A_q$) are located. The surface $A_q$ is in the $z_1z_2$ plane, see Fig \ref{fig:piezoelectriccomposite}, and the applied current acts in the normal of $A_q$, i.e. in the $z_3$ direction. Therefore, the current input is related to the current density through
    \begin{align}\label{eq:current_actuation}
    \begin{split}
        J_3(z_1)&=\lim\limits_{A_q\to 0}\frac{1}{A_q}\mathcal{I}(t)\\
                &=\frac{1}{2g_{b}}\mathcal{I}(t),
        \end{split}
    \end{align}
	and acts through the external work 
	\begin{align}\label{eq:workterm_current}
		\begin{split}
		W&=\int_{V}\frac{1}{2g_{b}}\mathcal{I}(t)\Phi dV\\
		&=\int_0^\ell (h_b-h_a)\mathcal{I}(t)\Phi dz_1,
	    \end{split}
	\end{align}
	as a homogeneous current source. \\
        
    Applying Hamilton's principle \cite{lanczos1970variational} to the definite integral 
    \begin{align*}\small
    \begin{split}
        {J}& :=\int_{t_{0}}^{t_{1}}  \mathcal{L}~d t\\
        &=\int_{t_{0}}^{t_{1}} \frac{1}{2} \int_{0}^{\ell} \rho\left(A \dot{v}^{2}+I \dot{w}_{z}^{2}-2 I_{0} \dot{v} \dot{w}_{z}\right)+\frac{1}{\beta} A \dot{\Phi}^{2}+\rho A \dot{w}^{2}\\
        &+2 \gamma \dot{\Phi}\left(A v_{z}-I_{0} w_{zz}\right)-C\left(A v_{z}^{2}+I w_{zz}^{2}-2 I_{0}v_{z} w_{zz}\right)\\
        &-\frac{1}{\mu} A \Phi_{z}^{2}+\left(h_{b}-h_{a}\right) I(t) \Phi ~d z_1 ~d t,
    \end{split}
\end{align*}
    where the Lagrangian \eqref{eq:Lagrangian_standard} is composed of the energies \eqref{eq:FD_energies}, work expression \eqref{eq:workterm_current}, and inertia's \eqref{eq:crossseactionalA} leads to the variations
    \begin{align}
         \begin{split}\label{eq:variations_actuator} 
        \delta J &:=\int_{t_0}^{t_{1}} \int_{0}^{l}\left(-p A \ddot{v}+p I_{0} \ddot{w}+C A v_{z_1 z_1}-CI_{0} w_{z_1 z_1 z_1}\right.\\&\left.\quad\quad-\gamma A \dot{\Phi}_{z_1} \right)\delta v+\left(-\rho I \ddot{w}_{z_1 z_1}+\rho I_{0} \ddot{v} + C I w_{z_1 z_1 z_1}\right.\\&\left.\quad\quad-C I_{0} v_{z_1 z_1 } +\gamma I_{0} \dot{\Phi}_{z_1 }\right) \delta w_{z_1} +(-\rho A \ddot{w}) \delta w \\
        &\quad\quad+\left(-\frac{1}{\beta} A \ddot{\Phi}-\gamma\left(A \dot{v}_{z_1}-I_{0} \dot{w}_{z_1 z_1}\right)\right.\\&\left.\quad\quad+\frac{1}{\mu_{1}} A \Phi_{z_1 z_1}+\left(h_{b}-h_{a}\right) I(t) \right) \delta \Phi ~d z_1 ~d t\\
        &\quad\quad +\left[\left(-CA v_{z_1}+C I_{0} v_{z_1 z_1}-\gamma A \dot{\Phi}\right) \delta v\right]_{0}^{l} \\
         &\quad\quad +\left[\left(-C I w_{z_1 z_1}+C I_{0} v_{z_1}-\gamma I_{0} \dot{\Phi}\right) \delta w_{z_1}\right]_{0}^{l}\\
        & \quad\quad+\left[\left(-\frac{1}{\mu} A \Phi_{z_1}+\gamma \left(A v_{z_1}-I_{0} w_{z_1 z_1}\right)\right) \delta \Phi\right]_{0}^{l}
            \end{split}
    \end{align}
     {with respect to the variations in the generalized coordinates $q=\col(v,w_z,\Phi)$} and using integration by parts (IBP). From \eqref{eq:variations_actuator}, we have that the dynamical PDE's describing the behaviour of the fully dynamic electromagnetic piezoelectric actuator, as follows
    	\begin{align}\label{eq:PDE_currentactuator_FD}\small
    	\begin{split}
    	&\rho\braces{A\ddot{v}-I_0\ddot{w}_{z_1}}=C_{11}\braces{Av_{z_1z_1}-I_0w_{z_1z_1z_1}}-\gamma A\dot{\Phi}_{z_1}\\
    	&\rho\braces{I\ddot{w}_{z_1}-I_0\ddot{v}}=C_{11}\braces{Iw_{z_1z_1z_1}-I_0v_{z_1z_1}}+\gamma I_0\dot{\Phi}_{z_1}\\
    	&\frac{1}{\beta}A\ddot{\Phi}=\frac{1}{\mu}A\Phi_{z_1z_1}-\gamma\braces{A\dot{v}_{z_1}-I_0\dot{w}_{z_1z_1}}-\braces{h_b-h_a}\mathcal{I}(t),
    	\end{split}
    	\end{align}
    	on the spatial domain $z_1\in[0,\ell]$, with essential boundary conditions $v(0)=\dot{v}(0)=0$, $w_{z_1}(0)=\dot{w}_{z_1}(0)=0$, $\Phi(0)=\dot{\Phi}(0)=0$ and natural boundary conditions $C_{11}A v_{z_1}(\ell)-C_{11}I_0w_{z_1z_1}(\ell)-\gamma A \dot{\Phi}(\ell)=0$, $C_{11}Iw_{z_1z_1}(\ell)-C_{11}I_0v_{z_1}(\ell)+\gamma I_0  \dot{\Phi}(\ell)=0$, $\frac{1}{\mu}A\Phi_{z_1}(\ell)-\gamma(A\dot{v}(\ell)-I_0\dot{w}_{z_1}(\ell))=0$. 
    	
    \blue{
        \begin{remark}\label{rem:on_the_EBBT}
            The dynamics \eqref{eq:PDE_currentactuator_FD} are derived using Hamilton's principle, where the variations are taken with respect to the variations in the generalized coordinates $q=\col(v,w_{z_1},\Phi)$ and produces the algebraic equation $\rho A\ddot{w} = 0$, see \eqref{eq:variations_actuator}.  Another approach would be using the generalized coordinates $\tilde{q}=col(v,w,\Phi)$, which result in the classical Euler-Bernoulli beam model with a fourth-order spatial derivative on $w$ and additionally the algebraic equations $\rho I\ddot{w}_{z_1}-\rho I_0 \ddot{v}=0$. The model using $\tilde{q}$ is derived from \eqref{eq:variations_actuator} by using IBP on the potential energy part associated with the variations $\delta w_{z_1}$.             The motivation of using the generalized coordinates $q$ comes from the deflection \eqref{eq:devormation_vvect}, produced by the assumption of the linear Euler-Bernoulli beam theory, given in the variables $v$ and $w_{z_1}$. The second row of \eqref{eq:PDE_currentactuator_FD} can be seen as a dynamical equations for the infinitesimal translation on the $z_1z_2-$plane. When the variations are taken with respect to $\tilde{q}$ the dynamical equation for the transverse displacement, i.e. contour, is obtained. The model with dynamical equation for the infinitesimal translation, using the variations $q$ has the following advantage; a third order spatial derivative on the transverse displacement is obtained, which corresponds to a first order spatial derivative with respect to the energy variable $w_{z_1z_1}$ and yields, therefore, a model with beneficial structure for analysis, while being able to determine the infinitesimal translation through $w_{z_1}$.\\
            With use of the Timoshenko beam theory, we can show that the dynamical equation for the infinitesimal translation of \eqref{eq:PDE_currentactuator_FD} satisfies the Euler-Bernoulli beam a priori assumptions, which we will touch upon here. The Euler-Bernoulli beam theory is based on the following three a priori assumptions \cite{carrera2011beam};
            \begin{enumerate}
                \item the cross-section is rigid on its plane,
                \item the cross-section rotates around a neutral surface remaining plane,
                \item the cross-section remains perpendicular to the neutral surface during deformation.
            \end{enumerate}
            The first two assumptions are shared with the Timoshenko beam theory \cite{carrera2011beam}, while the third assumptions is relaxed in a Timoshenko beam. Therefore, the Timoshenko beam model contains also the rotation of the beam's cross section $\phi(z_1,t)$ by means of shear stresses. Let $K$ denote the shear constant, then the Timoshenko beam equation can be written as,
        \begin{align}\label{eq:TimoshenkoBeam}
            \begin{split}
            \rho A \ddot{w}&=K\braces{\frac{\partial^2 w}{\partial z_1^2}-\frac{\partial \phi}{\partial z_1}}\\
            \rho I\ddot{\phi}&= C_{11}I \frac{\partial^2 \phi}{\partial z_1^2}-K\braces{\phi-\frac{\partial w}{\partial z_1}},
             \end{split}
        \end{align}
            see for instance \cite{carrera2011beam,Macchelli2004modelingTimoshenkobeam}. By imposing assumption three in \eqref{eq:TimoshenkoBeam}, using the constraint $\phi=w_{z_1}$, we obtain the dynamical equations
            \begin{align}\label{eq:TimoshenkoBeam_reduced}
                \begin{split}
                    \rho A \ddot{w}&=0\\
                    \rho I\ddot{w}_{z_1}&= C_{11}I \frac{\partial^3 w}{\partial z_1^3},
                 \end{split}
        \end{align}
            where the influence of shear stresses is mitigated.
            It can be seen that the first row of \eqref{eq:TimoshenkoBeam_reduced} is the algebraic equation in \eqref{eq:variations_actuator}.  Moreover, the second row of \eqref{eq:TimoshenkoBeam_reduced} returns in row two of \eqref{eq:PDE_currentactuator_FD}, as the mechanical beam part. Note that \eqref{eq:TimoshenkoBeam_reduced} is in the beam configuration and \eqref{eq:PDE_currentactuator_FD} in the actuator configuration. Therefore, we can make the similarity more clear by taking centroidal coordinates $h_b=-h_a$, such that $I_0\rightarrow 0$ to obtain the decoupled stretching and bending equations, i.e. the piezoelectric beam equations. Consequently, the derived equation for bending coincides with the deflecting beam equation in \eqref{eq:TimoshenkoBeam_reduced}. Taking centroidal coordinates will result in the dynamics of a piezoelectric beam. Letting $I_0\rightarrow 0$ can be done in the dynamical equations of piezoelectric actuators (e.g. \eqref{eq:PDE_currentactuator_FD}), or directly in the energies making up the Lagrangian (e.g. \eqref{eq:FD_energies}), see \cite{MenOSIAM2014} for a voltage actuated beam. 
        \end{remark}
        }
   
        The total energy of the actuator \eqref{eq:PDE_currentactuator_FD} is given by
    	\begin{align}\label{eq:Hamiltonian_current_FD}\small
    	\begin{split}
    	\mathcal{H}&_{FD}(t)=\frac{1}{2}\int_{0}^{\ell}\left[\rho\braces{A\dot{v}^2+I\dot{w}_{z_1}^2-2I_0\dot{v}\dot{w}_{z_1}}+\frac{1}{\beta}\dot{\Phi}^2\right.\\ &\quad\left.+C_{11}\braces{Av_{z_1}^2+Iw_{z_1z_1}^2-2I_0v_{z_1}w_{z_1z_1}}+\frac{1}{\mu}\Phi_{z_1}^2\right]dz_1,
    	\end{split}
    	\end{align}
    	obtained by summation of the energies \eqref{eq:FD_energies} and solving for the algebraic equation. \\
    
    \begin{remark}
    The external work \eqref{eq:workterm_current} is a direct consequence from Ampere's law \eqref{eq:Ampereslaw} and becomes evident by reducing \eqref{eq:Ampereslaw} for the case of the piezoelectric actuator. Then, we obtain the scalar equation,    
	    \begin{align}\label{eq:dynamicsimple}
		\frac{1}{\beta}\frac{\partial^2}{\partial t^2}\Phi&=\frac{1}{\mu}\frac{\partial^2}{\partial z_1^2}\Phi-\gamma\frac{\partial}{\partial t}(v_{z_1}-z_3w_{z_1z_1})-J_3,
	\end{align}
    where we made use of \eqref{eq:magneticfluxexpression1}, \eqref{eq:magneticfluxexpression2} and \eqref{eq:Constitutive_relations_scalar}. Combining \eqref{eq:dynamicsimple}, \eqref{eq:current_actuation} and integrating both sides with respect to the cross-section $A$, see \eqref{eq:crossseactionalA}, we obtain the same expression as the third equation of \eqref{eq:PDE_currentactuator_FD}. This shows the validity of the work term \eqref{eq:workterm_current}. 
    \end{remark}

    	The dynamics of the current actuated piezoelectric composite, depicted in Fig \ref{fig:piezoelectriccomposite}, is obtained by interconnecting the PDE \eqref{eq:PDE_currentactuator_FD} with equations associated with a purely mechanical substrate. As we have shown in Remark \ref{rem:on_the_EBBT} that the mechanical part of \eqref{eq:PDE_currentactuator_FD} reflects the Euler-Bernoulli beam theory, we continue by employing the same approach, as in \cite{Multi_MacchellivdSchaft2004,VenSSIAM2014}. 
    	Therefore, let the piezoelectric constant $\gamma\rightarrow 0$ in \eqref{eq:PDE_currentactuator_FD} and \eqref{eq:Hamiltonian_current_FD} and let the subscript $p$ and $s$ correspond to the piezoelectric actuator and substrate, respectively to define the coefficients
    	\begin{eqnarray}\label{eq:new_energycoefficients}
        		&\rho_A:=\rho_pA_p+\rho_s A_s,\quad  &C_A:=C_pA_p+C_sA_s, \nonumber \\
        		&\rho_I:=\rho_pI_p+\rho_s I_s, \quad &C_I:=C_pI_p+C_sI_s,\\
        		&\rho_{I_0}:=\rho_p I_0, \quad &C_{I_0}:=C_{11,p}I_0.\nonumber
    	\end{eqnarray} 
    	Then, with use of \eqref{eq:new_energycoefficients} the fully dynamic electromagnetic current actuated piezoelectric composite, depicted in Fig \ref{fig:piezoelectriccomposite}, is written as
    	\begin{align}\label{eq:PDE_currentcomposite_FD}\small
    	    \begin{split}
            	&\rho_A\ddot{v}-\rho_{I_0}\ddot{w}_{z_1}=C_A v_{z_1z_1}-C_{I_0}w_{z_1z_1z_1}-\gamma A_p\dot{\Phi}_{z_1}\\
            	&\rho_I \ddot{w}_{z_1}-\rho_{I_0}\ddot{v}=C_I w_{z_1z_1z_1}-C_{I_0}v_{z_1z_1}+\gamma I_0\dot{\Phi}_{z_1}\\
            	&\frac{1}{\beta}A_p\ddot{\Phi}=\frac{1}{\mu}A_p\Phi_{z_1z_1}-\gamma\braces{A_p\dot{v}_{z_1}-I_0\dot{w}_{z_1z_1}}-\braces{h_b-h_a}\mathcal{I}(t),
        	\end{split}
    	\end{align}
    	on the spatial domain $z_1\in[0,\ell]$, with essential boundary conditions as in \eqref{eq:PDE_currentactuator_FD} and natural boundary conditions $C_A v_z(\ell)-C_{I_0}w_{zz}(\ell)-\gamma A_p \dot{\Phi}(\ell)=0$, $C_Iw_{zz}(\ell)-C_{I_0}v_{z}(\ell)+\gamma I_0  \dot{\Phi}(\ell)=0$, $\frac{1}{\mu}A_p\Phi_z(\ell)-\gamma(A_p\dot{v}(\ell)-I_0\dot{w}_z(\ell))=0$.
    	
    	and total energy
    	\begin{align}\label{eq:Hamiltonian_PDE_fd_comp}\small
        	\begin{split}
            	\mathcal{H}(t)&=\frac{1}{2}\int_0^\ell\left[\rho_A \dot{v}^2+\rho_I\dot{w}^2_{z_1}-2\rho_{I_0}\dot{v}\dot{w}_{z_1}+\frac{1}{\beta}A_p\dot{\Phi}^2\right.\\ &\quad\left.+C_Av_{z_1}^2+C_Iw_{z_1z_1}^2-2C_{I_0}v_{z_1}w_{z_1z_1}+\frac{1}{\mu}A_p\Phi_{z_1}^2\right]dz_1.
        	\end{split}
    	\end{align} 
     This concludes the model derivation of the novel current actuated piezoelectric composite, which is analyzed further in the upcoming sections. Here we will continue by illustrating the similarities and differences with existing models in the literature.\\  
        
       The newly proposed piezoelectric actuator/composite model differs from the existing fully dynamic electromagnetic current actuated piezoelectric actuator/composite models, by treating Maxwell's equations differently.\\
        
        The treatment of Maxwell's equations that leads to the definition of the magnetic-flux density and subsequently to the novel dynamical equations for current actuated piezoelectric beams and composites is inspired by the treatment of Maxwell's equations for voltage actuated piezoelectric beams. See for instance \cite{MenOSIAM2014} for a detailed exposition. In \cite{MenOSIAM2014} a similar approach results in the definition of the charge by integrating the electric displacement, resulting in a voltage actuated piezoelectric beam. The quasi-static current actuated model presented in \cite{VenSSIAM2014} suggests the existence of a purely current actuated piezoelectric actuator/composite under the fully dynamic electromagnetic field assumption. In fact, with some effort, it can be shown that the fully dynamic current actuated system presented in this work and the current actuated quasi-static piezoelectric system presented in \cite{VenSSIAM2014} are related. More precisely, the system obtained by reducing the electromagnetic assumption to the quasi-static situation (i.e. let $\frac{\partial}{\partial z_1}B_2(=\Phi_{z_1z_1})\rightarrow 0$) in \eqref{eq:PDE_currentactuator_FD}, and linearization combined with the reduction of the beam theory (i.e. from the Timoshenko to Euler-Bernoulli beam theory) of the current actuated quasi-static piezoelectric system presented in \cite{VenSSIAM2014}, result in coinciding systems. The definition of the magnetic-flux \eqref{eq:magneticfluxexpression1} is crucial for the derivation of the current actuated fully dynamic piezoelectric actuator and composite model \eqref{eq:PDE_currentactuator_FD} \eqref{eq:PDE_currentcomposite_FD}.         In \cite{voss2011CDC}, it has been shown that the approximations of the quasi-static current actuated model does not satisfy the necessary condition for stabilizability. 

        
       The two existing principles to derive current actuated piezoelectric beams and composites result either from taking a charge actuated piezoelectric beam and add a dynamical equation on the boundary,
        or utilizing magnetic vector potentials $\Tilde{\vect{A}}$. The first approach results  in a current-through-the-boundary type model, by mathematically adding an integrator for the charge $Q$ on the boundary, i.e. $\mathcal{I}(t)=\tfrac{d}{dt}Q(t)$ and is employed in  \cite{OzerTAC2019,VenSSIAM2014}. More precisely, the fully dynamic current actuated system in \cite{VenSSIAM2014} exploits the boundary ports of the port-Hamiltonian formalism, mimicking the charge integrator used in \cite{OzerTAC2019}. Physically, either cases correspond to the use of some sort of electric circuitry. The charge and from there resulting current actuated systems have similar stabilizability properties. Besides, current-through-the boundary type systems can utmost asymptotically stabilize the system due to the bounded input operator, see \cite{OzerTAC2019}.\\  
        
        The alternative existing approach, resulting a purely current actuated system, uses magnetic vector potentials $\Tilde{A}$ as per Gauss's magnetic law \eqref{eq:GausssMagnetic} it is evident that there exist magnetic vector potentials, such that
        \begin{align}\label{eq:magvectpot}
	        \vect{B}=\nabla\times \vect{\Tilde{A}}.
	    \end{align}
        Substituting \eqref{eq:magvectpot} into Faraday's law \eqref{eq:Faraday'sLaw} results in an expression of the electric-field of a piezoelectric actuator as follows,
    	\begin{align}\label{eq:Efield}
    	    \vect{E}&=-\nabla {\varphi}-\frac{\partial }{\partial t}\vect{\tilde{A}},
    	\end{align}
    	where $\varphi$ denotes the electric scalar potential. As a result, the electric scalar potential ${\varphi}$ and magnetic vector potentials $\vect{\tilde{A}}$ are not uniquely defined \cite{MenOCDC2014}. Therefore, a gauge function, such as the Coulomb gauge or Lorentz gauge, is required to uniquely define $\varphi$ and $\vect{\tilde{A}}$. Although these gauge conditions do not influence the electric-field $\vect{E}$ and magnetic field $\vect{B}$, they do have their specific characteristics \cite{eom2013maxwell} and influence the dynamic equations governing the piezoelectric actuator, see for instance \cite{MenOCDC2014,OzerTAC2019}. In \cite{OzerTAC2019}, it has been shown that the purely current actuated fully dynamic piezoelectric system using electric vector potentials and a gauge condition lack the stabilizability property.\\ 
    	
        In our work, the definition of the magnetic-flux \eqref{eq:magneticflux_flow3} is pivotal. Our approach results in a purely current actuated piezoelectric system and circumvents the necessity of a gauge condition. Moreover, the new approach is analogous to the derivation of voltage actuated piezoelectric actuators, see for instance \cite{MenOSIAM2014} and appends the existing framework of modelling piezoelectric actuators. In the upcoming sections, we attempt to show the usefulness of the developed novel current actuated piezoelectric composite model by showing it is well-posedness and investigate the stabilizability properties for two approximation methods.

\section{Well-posedness}
    In this section, we show that the obtained dynamical system \eqref{eq:PDE_currentcomposite_FD} is well-posed. More specifically, we define an operator associated with the PDE \eqref{eq:PDE_currentcomposite_FD} and show, using the Lumer-Philips Theorem \cite{lumer1961}, that the operator is, in fact, a generator of a strongly continuous semigroup of contractions. 
    \begin{theorem}{\bf (Lumer-Phillips theorem)}\label{theorem:Lumer-Phillips}
    The closed and densely defined operator $A$ generates a strongly continuous semigroup of contractions $T(t)$ on $X$, if and only if both $A$ and its adjoint $A^*$ are dissipative, i.e.
    \begin{align}\label{disspative_operator}
    \begin{matrix}
    \angles{A \vect{x},\vect{x}}_X&\leq &0, \\
        \angles{A^*{\vect{x}},\vect{x}}_X&\leq &0.
    \end{matrix}
    \end{align}
    \begin{proof} For the proof, see \cite{lumer1961}.
    \end{proof}
\end{theorem}
        Hence, we establish the well-posedness in the sense of semigroup theory \cite{CurtainZwart1995introduction}. \\
        
        Let the length of the beam be $\ell=1$, with the spatial variable $z\in \left[0,\ell\right]$, and define $H^1_0(0,1):=\setof{f\in H^1(0,1)\mid f(0)=0}$, with $H^1(0,1)$ denoting the first order Sobolev space and let $L^2(0,1)$ denote the space of square integrable functions. Inspired by \eqref{eq:Hamiltonian_PDE_fd_comp}, define the linear space
	\begin{align*}
	\mathcal{X}=\left\{\vect{x}\in H^1_0(0,1)\times H_0^1(0,1)\times H^1_0(0,1)\right.\\
	\left.\times L^2(0,1) \times L^2(0,1) \times L^2(0,1) \right\}
	\end{align*}
	and inner product
	\begin{align}\small
	\begin{split}
	\angles{\vect{x},\vect{y}}_\mathcal{X}:&=\angles{\begin{bmatrix}	x_1 \\ x_2 \\ x_3 \end{bmatrix},\begin{bmatrix}	y_1 \\ y_2 \\ y_3 \end{bmatrix}}_{H^1}+\angles{\begin{bmatrix}	x_4 \\ x_5 \\ x_6 \end{bmatrix},\begin{bmatrix}	y_4 \\ y_5 \\ y_6 \end{bmatrix}}_{L^2}\\
	&=\int_{0}^{1}\left[C_Ax_1'y_1'+C_Ix_2'y_2'-C_{I_0}\braces{x_1'y_2'+x_2'y_1'}\right.\\
	&\quad\left.+\frac{1}{\mu}	A_px_3' z_3'+\rho_A {x}_4{y}_4+\rho_I{x}_5{y}_5\right.\\
	&\quad\left.-\rho_{I_0}\braces{{x}_4{y}_5+{x}_5{y}_4}+\frac{1}{\beta}A_p{x_6}y_6\right]dz,
	\end{split}
	\end{align}
	where the prime indicate the spatial derivative with respect to $z_1$. The inner product $\angles{.\ ,.}_\mathcal{X}$ induces the norm $\norm{\vect{x}}^2_{\mathcal{X}}=\angles{\vect{x},\vect{x}}_\mathcal{X}=2\mathcal{H}(t)$ on $\mathcal{X}$, see \eqref{eq:Hamiltonian_PDE_fd_comp}.
	For simplicity, denote the spatial variable $z:=z_1$, additionally let $\partial_{z}:=\frac{\partial}{\partial z}$, and define $\vect{x}:=\begin{bmatrix}v & w_{z} & \Phi & \dot{v} & \dot{w}_{z} & \dot{\Phi}	\end{bmatrix}^T$ to be the state, and current input $u(t)=\mathcal{I}(t)$.
	Then, the operator
	\begin{equation}\label{eq:opA_FD}\footnotesize
    	\begin{gathered}
        	\mathcal{A}:\Dom(\mathcal{A})\subset \mathcal{X}\rightarrow\mathcal{X},\\
            	\mathcal{A}=\left[
            	\begin{array}{ccc|ccc}
            	& & &1 & &\\
            	& && &1 &\\
            	& & & & &1\\ \hline
            	a_{41}\partial_{z}^2&-a_{42}\partial_{z}^2 & & & &-a_{46}\partial_{z}\\
            	-a_{51}\partial_{z}^2&a_{52} \partial_{z}^2& & & &a_{56}\partial_{z}\\
            	& & \tfrac{\beta}{\mu}\partial_{z}^2 &-\gamma\beta\partial_{z}&\gamma\beta\tfrac{I_0}{A_p}\partial_{z}&
            	\end{array}
            	\right]
    	\end{gathered}
	\end{equation}
	with 
	
	\begin{align}\label{eq:domain_A}\small
        \begin{split}
            	\Dom(\mathcal{A})=&\left\{\vect{x}\in \mathcal{X}\mid 	C_A x_1'(1)-C_{I_0}x_2'(1)-\gamma A_p x_6(1)=0, \right.\\
	    &~\left.C_Ix_2'(1))-C_{I_0}x_1'(1)+\gamma I_0  x_6(1)=0,\right.\\ 
	    &~\left.\frac{1}{\mu}A_p x_3'(1)-\gamma(A_p x_4(1)-I_0x_5(1))=0.\right\}
        \end{split}
	\end{align}

	and the coefficients of \eqref{eq:opA_FD} defined as
	\begin{eqnarray}\label{eq:coefficientsA}\small
    	a_{41}&:=\frac{\rho_I C_A - \rho_{I_0} C_{I_0}}{\rho_A\rho_I-\rho_{I_0}^2}, \quad\quad a_{51}&:=\frac{\rho_{A} C_{I_0}-\rho_{I_0} C_A  }{\rho_A\rho_I-\rho_{I_0}^2}, \nonumber\\
    	a_{42}&:=\frac{\rho_{I} C_{I_0}-\rho_{I_0} C_I }{\rho_A\rho_I-\rho_{I_0}^2}, \quad\quad 	a_{52}&:=\frac{\rho_{A} C_I - \rho_{I_0} C_{I_0}}{\rho_A\rho_I-\rho_{I_0}^2},  \nonumber\\
    	a_{46}&:=\gamma \frac{\rho_I A_p-\rho_{I_0}I_0}{\rho_A\rho_I-\rho_{I_0}^2}, \quad\quad  	a_{56}&:=\gamma \frac{\rho_{A}I_0-\rho_{I_0} A_p}{\rho_A\rho_I-\rho_{I_0}^2},  \nonumber \\
    	a_{63}&:= \frac{\beta}{\mu}, \quad a_{64}:=\gamma\beta, \quad\quad a_{65}&:=\gamma\beta\frac{I_0}{A_p}\quad\quad\quad\quad 
	\end{eqnarray}
	is densely defined in $\mathcal{X}$. Note that $K_2$ contains first order spatial derivative operators.
	Let the bounded input operator be
	\begin{align*}
	B&=\begin{bmatrix}
	0 & 0&  0& 0 & \frac{-\beta}{2g_b},
	\end{bmatrix}^T
	\end{align*}
	then, together with the operator defined in \eqref{eq:opA_FD} the state-space description of the set of PDEs \eqref{eq:PDE_currentcomposite_FD} can be written in short hand form, as follows
	\begin{align}\label{eq:FDstatespace}
	\dot{x}&=\mathcal{A}x+Bu(t),
	\end{align}
	with $\Dom(\mathcal{A})$ and $u(t)\in L^2(0,T)$. To establish the well-posedness of the operator \eqref{eq:opA_FD} in the sense of semigroup theory, we require the following useful Lemma. 
	
		\begin{lemma}\label{lem:FD_adjoint}
		The adjoint $\mathcal{A}^\ast$ of the operator $\mathcal{A}$, defined in \eqref{eq:opA_FD}, is skew-adjoint. More specifically,			\begin{align*}
		\mathcal{A}^\ast=-\mathcal{A},
		\end{align*}
		with $\Dom(\mathcal{A})=\Dom(\mathcal{A}^\ast)$.
	\end{lemma}
	
	\begin{proof}
		For any $\vect{u}=\begin{bmatrix}u_1 & \dots &u_6\end{bmatrix}^T$ and $\vect{v}=\begin{bmatrix}v_1 & \dots &v_6\end{bmatrix}^T$ $\in\Dom(\mathcal{A})$ we have,
		\begin{align}
		&\angles{\mathcal{A}\vect{u},\vect{v}}_\mathcal{X}
		=\int_{0}^{1}\left[u_1'\braces{C_Av_4'-C_{I_0}v_5'}\right.\nonumber\\ & ~\quad \left.+u_2'\braces{ C_I v_5'-C_{I_0}v_4'}\right.\left.+\tfrac{A_p}{\mu}u_3'v_6'\right.\nonumber\\ & ~\quad \left.+\braces{a_{41} u_1''-a_{42} u_2''-a_{46} u_6'}\braces{\rho_A v_4 -\rho_{I_0}v_5}\right.\nonumber \\ 
			&~\quad+\left.\braces{a_{52}u_2''-a_{51}u_1''+a_{56} u_6'}\braces{\rho_I v_5-\rho_{I_0}v_4}\right.\nonumber \\ 
		&~\quad+\left.\braces{a_{63}u_3''-a_{64} u_4'+a_{65} u_5'}\tfrac{A_p}{\beta}v_6 \right]d_z\nonumber \\ 
		&\quad=-\int_{0}^{1}\left[\braces{C_A u_1'-C_{I_0} u_2'}v_4'\right.\nonumber\\ & \quad \left.+  \braces{C_I u_2'-C_{I_0} u_1'}v_5' \right.\left.+\tfrac{A_p}{\mu}u_3'v_6'\right.\nonumber\\ & \quad \left.+\braces{\rho_A u_4 -\rho_{I_0}u_5}\braces{a_{41} v_1''-a_{42} v_2''-a_{46} v_6'}\right.\nonumber \\ 
		&~\quad+\left.\braces{\rho_I u_5-\rho_{I_0}u_4}\braces{a_{52} v_2''-a_{51}v_1''+a_{56} v_6'}   \right.\nonumber \\ 
		&~\quad+\left. \tfrac{A_p}{\beta}u_6\braces{a_{63}v_3''-a_{64} v_4'+a_{65}v_5'}	 \right]d_z\nonumber \\ 
		&\quad=\angles{\vect{u},-\mathcal{A}\vect{v}}_\mathcal{X},
		\end{align}
		where we used the boundary conditions $v_1(0)=v_2(0)=v_3(0)=0$ and 	
	$C_A v_1'(1)-C_{I_0}v_2'(1)-\gamma A_p v_6(1)=0$, $C_Iv_2'(1))-C_{I_0}v_1'(1)+\gamma I_0  v_6(1)=0$, 	$A_pv_3'(1)-\gamma(A_p v_4(1)-I_0v_5(1))=0$.
	Moreover, we have that $v_1,v_2,v_3\in H^1_0(0,1)$ and $v_4,v_5,v_6\in L^2(0,1)$, therefore we define the domain of $\mathcal{A}^\ast$ as
		\begin{align}\label{eq:domain_Astar}\small
		    \begin{split}
		        \Dom(\mathcal{A^\ast})=&\left\{\vect{v}\in\mathcal{X}\mid 	C_A v_1'(1)-C_{I_0}v_2'(1)-\gamma A_p v_6(1)=0, \right.\\
	        &~\left.C_Iv_2'(1))-C_{I_0}v_1'(1)+\gamma I_0  v_6(1)=0,\right.\\ &~\left.\frac{1}{\mu}A_pv_3'(1)-\gamma(A_p v_4(1)-I_0v_5(1))=0.\right\},
		    \end{split}
		\end{align}
		to conclude that $\Dom(\mathcal{A^\ast})=\Dom(\mathcal{A})$.
	\end{proof}

	Now we can establish the well-posedness of the novel current actuated piezoelectric composite in the absence of control.
	\begin{theorem}\label{theorem:contractionFD} The operator $\mathcal{A}$, defined in \eqref{eq:opA_FD}, generates a semigroup of contractions, satisfying $\norm{T(t)}\leq 	1$ on $\mathcal{X}$.
	\end{theorem}

	\begin{proof}
		The closed and densely defined operator $\mathcal{A}$ satisfies
		\begin{align*}\tiny
    		&\angles{\mathcal{A}\vect{z},\vect{z}}_\mathcal{X}=\int_0^1\bigg[
    		-a_{46}z_6'\braces{\rho_A z_4-\rho_{I_0}z_5}\\
    		 &\quad + a_{56}z_6'\braces{\rho_I z_5-\rho_{I_0}z_4}+\braces{-a_{64}z_4'+a_{65}z_5'}\tfrac{A_p}{\beta}z_6\bigg]dz\\
    		& \quad  +\left[ z_4(C_Az_1'-C_{I_0}z_2')+ z_5(C_I z_2'-C_{I_0}z_2')\right.\\
    		& \quad \left.+z_6(\tfrac{A_p}{\mu}z_3) \right]_0^1=0 \leq 0,\\
    		&\angles{\mathcal{A}^\ast\vect{z},\vect{z}}_\mathcal{X}=-\int_0^1\bigg[
    		-a_{46}z_6'\braces{\rho_A z_4-\rho_{I_0}z_5} \\
    		&\quad+ a_{56}z_6'\braces{\rho_I z_5-\rho_{I_0}z_4}
    		+\braces{-a_{64}z_4'+a_{65}z_5'}\tfrac{A_p}{\beta}z_6\bigg]dz\\
    		&\quad-\left[ z_4(C_Az_1'-C_{I_0}z_2')+ z_5(C_I z_2'-C_{I_0}z_2')\right.\\
    		& \quad \left.+z_6(\tfrac{A_p}{\mu}z_3) \right]_0^1=0 \leq 0,
		\end{align*}
		by straightforward calculations using IBD and the domains \eqref{eq:domain_A} and \eqref{eq:domain_Astar}. By the Lumer-Phillips Theorem \ref{theorem:Lumer-Phillips}, the operator generates a semigroup of contractions.
	\end{proof}
	
    This concludes the well-posedness of the  proposed current actuate piezoelectric composite. In the next section we investigate the stabilizability properties of the approximations of the piezoelectric composite \eqref{eq:opA_FD}.


\section{Approximations of piezoelectric composites}
   For control, simulation, and analysis purposes, it is useful to approximate a system governed by PDEs. In this section, we derive the approximated ordinary differential equations (ode) of the piezoelectric composite  \eqref{eq:opA_FD}. More specifically, we derive two ode systems using two approximation schemes, i.e. the finite element method (FEM) \cite{FlahertyNotes} and the mixed finite element method (MFEM) \cite{GoloSchaft2004}. During the approximation method, certain properties belonging to the original PDE may be lost. Therefore we investigate in the next section the stabilizability properties of the derived approximations of the two approximation schemes and see if the derived approximations bear different stabilizability properties, for this specific case.\\
    
     The application of FEM and MFEM result in a finite dimensional system of ordinary differential equations
        \begin{align}\label{eq:sequenceofapprox}
			\dot{\vect{x}}_N(t)=A_N\vect{x}_N(t)+B_Nu(t), \quad t>0, \quad \vect{x}_N\in \R^N,
		\end{align}
    where the number of ode's are determined by the number of segments $N$ considered, such that $A_N\in\R^{N\times N}$ and $B_N\in\R^N$. More specifically, the whole domain $\Omega$ is the union of all $N$ local segments $\Omega^j_{ab}=\brackets{a,b}=\brackets{z_{j-1},z_j}$ for $j\in\setof{1,\dots,N}$. The dimension of $X_N$ tends to infinity as $N$ tends to infinity, making $\Omega^j_{ab}$ infinitesimal small. \\
    
    First we derive the ode system using FEM, which can be directly done from \eqref{eq:opA_FD}. Subsequently, we derive the ode system using MFEM, where we first have to rewrite the PDE system \eqref{eq:opA_FD}. We conclude this section by verifying the approximations.

    \subsection{FEM approximation}
            For the FEM approximations of the fully dynamic current actuated piezoelectric composite \eqref{eq:opA_FD} on the local segment 
        define the vector $\vect{c^j}:=\begin{bmatrix} c^j_0 & c^j_1 \dots c^j_N \end{bmatrix}^*\in\R^N$, denoting the coefficients of the test functions for the $j$-th variable in the state $\vect{x}$, thus $j\in \setof{1,\dots, 6}$. Then,  the $N$-th order approximation of the linear fully dynamic piezoelectric composite can be written in matrix form \eqref{eq:approx_FEM},
         \begin{figure*}[b]
         \hrulefill
         \centering
    		\begin{align}\label{eq:approx_FEM}\tiny
    		\underbrace{\begin{bmatrix}
    			\dot{\vect{c}}^1 \\ \dot{\vect{c}}^2 \\ \dot{\vect{c}}^3 \\ \dot{\vect{c}}^4 \\ \dot{\vect{c}}^5 \end{bmatrix}}_{\dot{\vect{x}}_N}=
    		\underbrace{\begin{bmatrix}
    			& & & I_N & &\\ 
    			&&&& I_N &\\
    			&&&& &I_N \\
    			-a_{41}  M_1^{-1}K_2 & a_{42} M_1^{-1}{K}_2 &&&& a_{46} M_1^{-1}K_1 \\
    			a_{51} M_1^{-1}K_2 & -a_{52} M_1^{-1}{K}_2 &&&&  -a_{56} M_1^{-1} K_1 \\
    			&&-\frac{\beta}{\mu} M_1^{-1}K_2& \gamma\beta M_1^{-1} K_1 & -\gamma\beta \frac{I_0}{A_p} M_1^{-1}K_1\\
    			\end{bmatrix}}_{A_N}\underbrace{\begin{bmatrix}
    			\vect{c}^1 \\ \vect{c}^2 \\ \vect{c}^3 \\ \vect{c}^4 \\ \vect{c}^5  \\ \vect{c}^6\end{bmatrix}}_{\vect{x}_N}
    		+\underbrace{\begin{bmatrix}
    			0\\0 \\ 0 \\ 0 \\0\\ B_1 
    			\end{bmatrix}}_{B_N}\mathcal{I}(t),
    		\end{align} 
    		 \end{figure*} 
    		where the coefficients are as in \eqref{eq:coefficientsA} and the squared matrices $M_1,K_1,K_2\in\R^{M\times M}$ and column vector $B_1\in\R^M$ are composed of the local element matrices  as follows, 
    		
    		\begin{align}\label{eq:FEM_elementmatrices}
    		\tiny
    		\begin{split}
    		M_1&=\frac{h_j}{6}\begin{bmatrix} 4 &1& \\ 1 &4&1 \\ &1&\ddots&1 \\
    		&&2&4&2\\&&& 2&4
    		\end{bmatrix}, \ 	K_1=\frac{1}{2}\begin{bmatrix} 0 &-1& \\ 1 &0&-1 \\ &1&\ddots&-1 \\
    		&&1&0&-1\\&&& 1&0
    		\end{bmatrix}, \\ K_2&=\frac{1}{h_j}\begin{bmatrix} 2 &-1& \\ -1 &2&-1 \\ &-1&\ddots&-1 \\
    		&&-1&2&-1\\&&& -1&2
    		\end{bmatrix},\ B_1=\frac{-\beta}{2g_b}\begin{bmatrix}
    		1 \\ 1 \\\vdots \\1 \end{bmatrix}.
    			\end{split}
    		\end{align}
        This concludes the approximation of the fully dynamic piezoelectric composite using the finite element method. 

    \subsection{MFEM approximation}

    In this section we treat the approximation of the proposed current actuated piezoelectric cantilever piezoelectric beam using the MFEM method  \cite{GoloSchaft2004}. Therefore, we need to rewrite the PDE in a specific form, which we touch upon first. In the derivation of the current actuated model \eqref{eq:PDE_currentcomposite_FD} we made use of the Lagrangian $\mathcal{L}=\int_0^\ell {L(q,\dot{q})}dz_1$, where ${L}$ denotes the Lagrangian density function is a function of the generalized coordinates (recall, $q=\col(v,w_{z_1},\Phi)$) and the generalized velocities $\dot{q}$. To apply MFEM, \eqref{eq:PDE_currentactuator_FD} needs to be written in the form. 
    \begin{align}\label{eq:ph_desiredform}{}
        \frac{d}{dt}\begin{pmatrix} q_z \\ p \end{pmatrix}{}=\begin{pmatrix}\frac{\partial}{\partial z}\frac{\partial H }{\partial p}(q_z,p)\\\frac{\partial}{\partial z}\frac{\partial H }{\partial q_z}(q_z,p)
        \end{pmatrix}+Bu(t),
    \end{align}{}
    where $p$ denotes the generalized momenta and $H(q_z,p)$ is the density functional of the total energy \eqref{eq:Hamiltonian_PDE_fd_comp} in different coordinates $q_z=\col(v_z,w_{zz},\Phi_z)$ and $p$. This is done by using the Legendre transformation
    \begin{align}\label{eq:Legendre_trans}
        p:=\frac{\partial L}{\partial \dot{q}}(q,\dot{q})&=\begin{bmatrix}\rho_A \dot{v}-\rho_{I_0}\dot{w}_z \\ \rho_I \dot{w}_z-\rho_{I_0}\dot{v}\\ \frac{1}{\beta}A_p\dot{\Phi}+\gamma\braces{A_p\dot{v}-I_0\dot{w}_z}\end{bmatrix},\\
        H(q,p)&=p^T\dot{q}-L(q,\dot{q}).
    \end{align}{}
    Moreover, we require $H(q,p)\rightarrow H(q_z,p)$. Therefore, let $\bar{x}=\col(q,p)$, $x=\col(q_z,p)$ and consider the coordinate transformation  $x(q_z,p)=S\bar{x}(q,p)$, then
    \begin{align}\label{eq:Ham_coordinate_transform}
        H(q_z,p)=\bar{x}^T\bar{Q}\bar{x}=x^T S^{-T} \bar{Q} S^{-1} x = x^T Q x,
    \end{align}{}
    with 
    \begin{align}\label{eq:ph_Qmatrix}\tiny
        Q:&=S^{-T} \bar{Q} S^{-1}=\begin{bmatrix}Q_1 & Q_2 \\Q_2^T & Q4 \end{bmatrix},
        \end{align}
        where
        \begin{align*}
            Q_1&=\begin{bmatrix}
            C_A+\gamma^2\beta A_p & -\braces{C_I+\gamma^2\beta I_0}& 0\\
            -\braces{C_I+\gamma^2\beta I_0}& C_I+\gamma^2\beta \tfrac{I_0^2}{A_p}& 0\\
            0&0&\frac{A_p}{\mu}
            \end{bmatrix},\\
             Q_2&=\begin{bmatrix} 
                0 & 0 & -\gamma\beta\\
                0 & 0 & \gamma\beta\tfrac{I_0}{A_p}\\
                0&0&0
            \end{bmatrix},\\
            \ Q_4&=\begin{bmatrix}
                \tfrac{\rho_I}{\rho_A\rho_I-\rho_{I_0}^2}& \tfrac{\rho_I}{\rho_A\rho_I-\rho_{I_0}^2} &0\\
                 \tfrac{\rho_I}{\rho_{I_0}\rho_I-\rho_{I_0}^2}& \tfrac{\rho_A}{\rho_A\rho_I-\rho_{I_0}^2} &0\\
                0&0&\tfrac{\beta}{A_p}
            \end{bmatrix}.
        \end{align*}
        

\begin{figure*}[b]
\hrulefill
\centering
\begin{align}\label{eq:approx_MFEM}\tiny
    \begin{bmatrix}\dot{\vect{x}}_1 \\ \vdots \\ \dot{\vect{x}}_j \\ \vdots \\ \dot{\vect{x}}_N\end{bmatrix} = \begin{bmatrix} J_1 & -B^{+}_{2}B^{T+}_{1} & \dots\\
    \\
    \dots &  -B_j^{-}D_{j-1}B_{j-2}^{T-} & B_j^{-}B_{j-1}^{T-} & J_j & &-B_j^{+}B_{j+1}^{T+} & B_j^{+}D_{j+1}B_{j+2}^{T+} & \dots\\
    \\
     &&&&&& \dots & B_N^{-}B_{N-1}^{T-} &J_N 
    \end{bmatrix}+\begin{bmatrix}B^e \\ \vdots \\ B^e \\ \vdots \\B^e \end{bmatrix}u(t)
\end{align}{}
\end{figure*}

    The coordinate transformation \eqref{eq:Ham_coordinate_transform}, guarantees the same rate of change of the total energy, i.e.
    \begin{align}\label{eq:Hamiltonian_ph_rateofchange}\small
        \begin{split}
               \frac{d}{dt}H&=\Big[\underbrace{\braces{C_A v_z(z)-C_{I_0}w_{zz}(z)-\gamma A_p \dot{\Phi}(z)}}_{e^{B_1}_\ast}\underbrace{\dot{v}(z)}_{f^{B_1}_\ast} \\
        &+ \underbrace{\braces{C_Iw_{zz}(z)-C_{I_0}v_{z}(z)+\gamma I_0  \dot{\Phi}(z)}}_{e^{B_2}_\ast}\underbrace{\dot{w}_z(z)}_{f^{B_2}_\ast} \\
        &+ \underbrace{\braces{\frac{1}{\mu}A_p\Phi_z(z)-\gamma(A_p\dot{v}(z)-I_0\dot{w}_z(z))}}_{e^{B_3}_\ast}\underbrace{\dot{\Phi}(z)}_{f^{B_3}_\ast}  \Big]_0^\ell,
        \end{split}{}
    \end{align}{}
    where $e^{B_l}_*$, $f^{B_l}_*$, with $l\in\setof{1,2,3}$ denotes the port-Hamiltonian flows and efforts on the boundary $*\in\setof{0,\ell}$, see for instance \cite{port-HamiltonianIntroductory14,GoloSchaft2004,VenSSIAM2014}. \\

   Let $B^e=\begin{bmatrix}0 & 0 & 0& 0& 0 & -(h_b-h_a)\end{bmatrix}^T$ and denote the $i\times i$ identity matrix by $\mathbb{I}_i$. Then, by the quadratic nature of $H(q_z,p)$ we write the dynamics \eqref{eq:ph_desiredform} in the compact form
    \begin{align}\label{eq:ph_PDE_comp}
        \dot{x}&=\underbrace{\begin{bmatrix} 0 & \mathbb{I}_3 \\ \mathbb{I}_3 & 0\end{bmatrix}\frac{\partial}{\partial z}}_{J}Qx+B^eu,
    \end{align}{}
     with boundary conditions as boundary flows and efforts 
    \begin{align}\label{eq:ph_boundaryconditions}
       \begin{split}
            f^{B_1}_0=f^{B_2}_0=f^{B_3}_0=0,\\
          e^{B_1}_\ell=e^{B_2}_\ell=e^{B_3}_\ell=0
       \end{split}{}
    \end{align}{}
    defined in \eqref{eq:Hamiltonian_ph_rateofchange}, where the first row constitute the essential boundary conditions and the second row contains the natural boundary conditions.\\

 
 The piezoelectric composite \eqref{eq:ph_PDE_comp}, with \eqref{eq:ph_Qmatrix} and boundary conditions \eqref{eq:ph_boundaryconditions} is approximated using the mixed trial functions 
  \begin{align}\label{eq:MFEM_approximation_forms}
       \begin{split}
            x^i(z,t) &\approx x_{ab}^i(t)\phi_{ab}(z)\\
            \dot{x}^i(z,t) &\approx \dot{x}^i_{ab}(t)\phi_{ab}(z)\\
            e^i(z,t) &\approx e_{a}^i(t)\phi_a(z)+e_{b}^i(t)\phi_b(z),
       \end{split}{}
    \end{align}{}
    on the local domain $\Omega_{ab}$, for $j\in\setof{1,\dots,6}$ corresponding to the sequence of states in the original PDE. Note that the $e:=Qx$ in \eqref{eq:ph_PDE_comp} requires more smoothness in the spatial variable than $x$ or $\dot{x}$. Let $\vect{u}_{ab}=col(\vect{f}^B_a,\vect{e}_b^B)$, $\vect{y}_{ab}=col(\vect{e}^B_a,-\vect{f}_b^B)$ and following the procedure described in \cite{GoloSchaft2004} for \eqref{eq:ph_PDE_comp},
     with 
    \begin{align*}
        \phi_{ab}(z)=\frac{1}{b-a}, \quad \phi_a(z)=\frac{b-z}{b-a}, \quad \phi_b(z)=\frac{z-a}{b-a}
    \end{align*}
 results in the local finite dimensional system
\begin{align}
    \begin{split}
            \dot{\vect{x}}_{ab}&=J_{ab}(Q_{ab}\vect{x}_{ab})+B_{ab}\vect{u}_{ab}\\
    \vect{y}_{ab}&=B_{ab}^T(Q_{ab}\vect{x}_{ab})+D_{ab}\vect{u}_{ab}
    \end{split}{}
\end{align}{}
with 
\begin{align}
    \begin{split}
        J_{ab}&=2\begin{bmatrix} 0 & \mathbb{I}_3 \\  -\mathbb{I}_3  & 0  \end{bmatrix},\quad Q_{ab}=\frac{1}{b-a}Q,\\
        B_{ab}&=2\begin{bmatrix} 0 & -\mathbb{I}_3 \\  \mathbb{I}_3  & 0  \end{bmatrix}, \quad D_{ab} = \begin{bmatrix} 0 & -\mathbb{I}_3 \\ \mathbb{I}_3 & 0\end{bmatrix}.
    \end{split}{}
\end{align}{}
Then, interconnecting the system $j=1,\dots,N$ using $u_j=-B^{T-}y_{j-1}-B^{T+}y_{j+1}$ where $B^{T-}+B^{T+}=B_{j}$ is decomposed into the left and right component, with respectively the annotation $-,+$. Then, the interconnection of $N$ segments results in the global dynamical system \eqref{eq:approx_MFEM}, where $u(t)$ is the external current input. This concludes the approximation of the fully dynamic piezoelectric composite using the mixed finite element method.

\subsection{Verification of approximation}
   The approximated current-controlled piezoelectric composites models \eqref{eq:approx_FEM} and \eqref{eq:approx_MFEM} are verified in two manners using the system coefficients presented in Table \ref{tab:numerical_coefficients}. The first is the convergence of the tip response, i.e. the longitudinal and transverse deflection. The transverse deflection $w$ is numerically computed using the trapezium integration rule on $w_{z_1}$. In Fig \ref{fig:Validation_displacement}, both the transverse and longitudinal deflection of the actuated beam are depicted. It can be seen that the two approximations provide the same dynamics and converge to the same response. The second verification method underpins the observation of convergent trajectories and is by investigating the convergence of the imaginary part of the eigenvalues. Table \ref{tab:Validation_eigenvalues} summarizes the convergence for the first four eigenvalues (responsible for the dominant behaviour) for increasing approximation order for the two approximation methods.\\
    
   The presented data shows (under increasing approximation order) the convergence of both approximation methods individually, and it shows that the two approximation models converge towards a similar behaviour. 
    
    \begin{table}
        \centering
        \begin{tabular}{ l l l l }
            \hline
            Description & Symbol & Value & Unit \\ \hline
             \multicolumn{4}{l}{\textbf{Beam geometry}}\\ \hline
                 \hspace{3mm} Length        & $\ell$        & $1$     & $[m]$ \\
                 \hspace{3mm} Width         & $g_b$         & $0.1$   & $[m]$ \\
                 \hspace{3mm} Thickness     & $h$           & $0.01$  & $[m]$ \\ \hline 
                \multicolumn{4}{l}{\textbf{Beam Substrate}}\\ \hline
                \hspace{3mm} Volumetric density & $\rho$      & $5000$ & $[kg/m^3]$ \\
                \hspace{3mm} Young's Modulus    & $C_{11,s}$    & $10^5$ & $[N/m^2]$ \\ \hline
                \multicolumn{4}{l}{\textbf{Piezoelectric layer}}\\ \hline
                \hspace{3mm} Volumetric density         & $\rho_p$      & $7600$ & $[kg/m^3]$ \\
                \hspace{3mm} Young's Modulus            & $C_{11,p}$    & $140\times10^5$ & $[N/m^2]$ \\
                \hspace{3mm} Coupling coefficient       & $\gamma$      & $10^{-3}$    & $[C/m^2]$ \\
                \hspace{3mm} Dielectric impermittivity  & $\beta$       & $10^6$        & $[m/F]$ \\
                \hspace{3mm} Magnetic permeability      & $\mu$         & $1.2\times 10^6$ & $[H/m]$ \\
            \hline
        \end{tabular} \vspace{1mm}
        \caption{Physical coefficients}
        \label{tab:numerical_coefficients}
    \end{table}{}

    \begin{remark}
        From Table \ref{tab:Validation_eigenvalues}, it can be seen that the MFEM method approximates the eigenvalues higher than the FEM method. However, it does so with higher matrix sparsity, which provides a computational advantage. With increasing approximation order, the eigenvalues of the two approximation methods converge towards each other.
    \end{remark}
    
    \begin{figure*}
		\includegraphics[width=\linewidth]{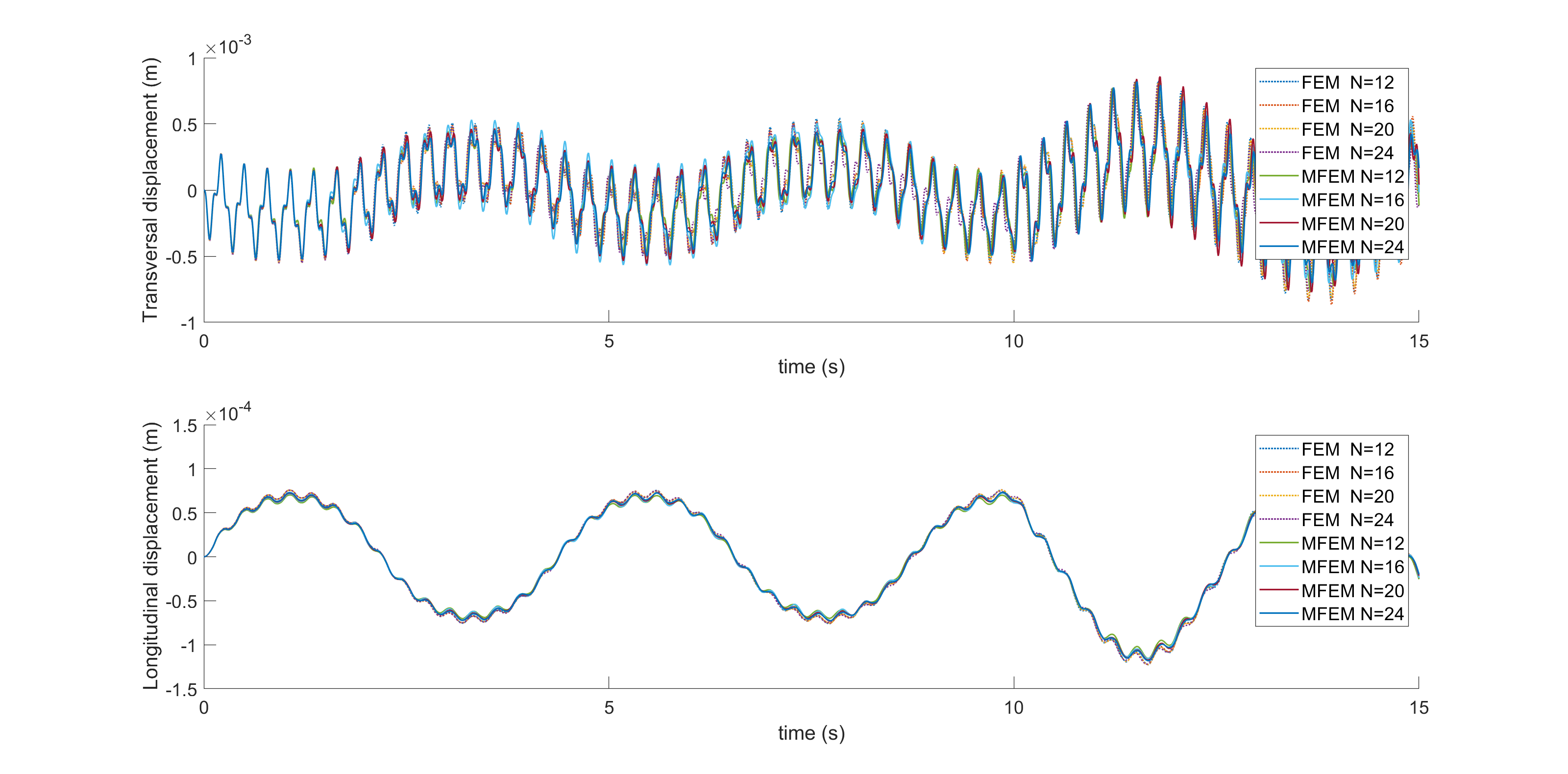}
		\caption{Converging transverse and longitudinal displacement for increasing approximation order $N=\begin{bmatrix}12 & 16&20&24\end{bmatrix}$.} 
		\label{fig:Validation_displacement}
	\end{figure*}
    
   \begin{table}[!t]
        \caption{Convergence of imaginary part of numerical determined eigenvalues.}
\begin{center}
   \resizebox{\columnwidth}{!}{
            \begin{tabular}{| r || l | l || l | l || l | l || l | l |}
        \hline
           & \multicolumn{2}{l||}{$\im(\lambda_1)$} & \multicolumn{2}{l||}{$\im(\lambda_2)$} & \multicolumn{2}{l||}{$\im(\lambda_3)$} & \multicolumn{2}{l|}{$\im(\lambda_4)$}  \\  \hline
         N & FEM & MFEM  & FEM & MFEM  & FEM & MFEM  & FEM & MFEM\\ \hline  \hline
    12 &  1.4350 & 1.4360 & 4.3295 & 4.3580 & 7.2982 & 7.4371 & 10.3913 & 10.8043 \\
    16 &  1.4345 & 1.4351 & 4.3174 & 4.3332 & 7.2419 & 7.3172 & 10.2360 & 10.4522 \\
    20 &  1.4343 & 1.4347 & 4.3118 & 4.3218 & 7.2158 & 7.2633 & 10.1644 & 10.2982 \\
    24 &  1.4342 & 1.4344 & 4.3087 & 4.3157 & 7.2017 & 7.2343 & 10.1255 & 10.2169 \\
    28 &  1.4341 & 1.4343 & 4.3069 & 4.3120 & 7.1932 & 7.2171 & 10.1022 & 10.1686 \\
    32 &  1.4341 & 1.4342 & 4.3057 & 4.3096 & 7.1877 & 7.2059 & 10.0870 & 10.1375 \\
    36 &  1.4340 & 1.4342 & 4.3049 & 4.3080 & 7.1839 & 7.1982 & 10.0766 & 10.1163 \\
    40 &  1.4340 & 1.4341 & 4.3043 & 4.3068 & 7.1812 & 7.1928 & 10.0692 & 10.1012 \\
   100 &  1.4339 & 1.4340 & 4.3022 & 4.3026 & 7.1715 & 7.1734 & 10.0426 & 10.0477 \\
        \hline
        \end{tabular} }\vspace{1mm}
        \label{tab:Validation_eigenvalues}
    \end{center}
    \end{table}

\section{Stabilizability of piezoelectric composite approximations}
    The stabilizability properties of the derived and verified ode systems of the approximated current-controlled piezoelectric composites are investigated in this section. Both the approximations using FEM or MFEM are finite-dimensional linear systems of the form 

    \begin{align}\label{eq:fin-dim_generalsys}
        \begin{split}
            \dot{x}(t)&=Ax(t)+Bu(t), \quad x\in\R^n\\
            y(t)&=Cx(t)
        \end{split}{}
    \end{align}
    with $u=\mathcal{I}(t)$, $y\in\R^p$, matrixes $A, B, C$ of suitable dimension, and $n=6N$. The controllability of such system can be verified using Kalman's controlability rank condition, where the rank of the controllability matrix $\mathcal{C}:=\begin{bmatrix}B&AB&\dots&A^{n-1}B\end{bmatrix}$ must be equal to the dimension of the state space, i.e. $\rank(\mathcal{C})=n$.\\
    
    The stabilizability of the system can be investigated by checking the 
    the controllable canonical form. Therefore, let $x_c(t)$ and $x_c(t)$ denote respectively the collection of controllable and uncontrollable states and define $\bar{x}(t):=\col(x_c(t),x_u(t))$. If there exist a non-singular state transformation matrix $T$ such that $\bar{x}=Tx$, we can obtain the system
    \begin{align}\label{eq:stabilizable_form}
       \begin{split}
            \begin{bmatrix}\dot{x}_c (t)\\ \dot{x}_u(t)\end{bmatrix}&=\underbrace{\begin{bmatrix}\bar{A}_{11} & \bar{A}_{12} \\ 0 & \bar{A}_{22}\end{bmatrix}}_{\bar{A}}\underbrace{\begin{bmatrix}{x}_c(t) \\ {x}_u(t)\end{bmatrix}}_{\bar{x}(t)}+\underbrace{\begin{bmatrix}\bar{B}_1 \\ 0\end{bmatrix}}_{\bar{B}}u(t)\\
             y(t)&=\bar{C}\bar{x}(t),
       \end{split}{}
    \end{align}{}
    isomorphic to \eqref{eq:fin-dim_generalsys}, with  $\bar{A}:=T^{-1}AT$, $\bar{B}:=T^{-1}B$, and $\bar{C}:=CT$. \\
    
    
    By investigating the ode systems symbolically, it can be shown that for $N=1$ and $N=2$ both FEM and MFEM approximations are controllable. However, for higher-order approximations ($N\geq3$), the cost to calculate and analyse $\mathcal{C}$ for FEM and MFEM becomes too expensive. Hence, for both cases, it was not possible to establish results for  $N\geq3$. The calculations have been performed with the use of the Peregrine high-performance cluster.\\
    
    We can however establish the controllability/stabilizability property for the current actuated piezoelectric composite for a specific case (specified coefficients) combined with simulation results. 
    From a more analytical point of view, we can investigate if the ode's satisfy Brockett's necessary condition for stabilizability \cite{Brockett83asymptoticstability}, which we will treat next.\\
    
	\begin{theorem}[Brockett condition]\label{thm:brockett}
			Consider the system $\dot{x}=f(x,u)$, where $x\in\R^n$. A necessary condition for the existence of a continuous feedback law $u=u(x)$ rendering $x_0\in\R^n $ locally asymptotically stable for the closed-loop system $\dot{x}=f(x,u(x))$ is that 
			\begin{align*}
			f(x,u)=y,
			\end{align*}
			is solvable for all $\|y\|$ sufficiently small.
	\end{theorem}
		For linear systems we derive the following corollary
	\begin{corollary}\label{thm:brockett_linear}
	    Consider the system $\dot{x}=Ax+Bu$, where $x\in\R^n$. A necessary condition for the existence of a continuous feedback law $u=Fx$ rendering $x_0\in\R^n $ locally asymptotically stable for the closed-loop system $\dot{x}=(A+BF)x$ is that 
	    \begin{align*}
		    \im\begin{bmatrix}A & B\end{bmatrix} \in \R^{n}.
		\end{align*}{}
	\end{corollary}
	This leads us to the following two results.
	
	\begin{theorem}\label{theorem:FD_FEM_stabilizablity}
			The approximated linear fully dynamic electromagnetic piezoelectric composite \eqref{eq:approx_FEM} using FEM, satisfies Brockett's necessary condition for asymptotic stabilizability through continuous state feedback. 
			\begin{proof}
			    Using \eqref{eq:approx_FEM} we compute
			    \begin{align*}
				\begin{bmatrix}	A_N & B_N  \end{bmatrix}&\backsim \begin{bmatrix}
			    I_{6N} & \tilde{B}_{FEM}
				\end{bmatrix},
				\end{align*}
			    with 
			    \begin{align*}
			         \tilde{B}^T_{FEM} = \begin{bmatrix}0_{1\times 2N} & \Asterisk_{1\times N} & 0_{1\times 3N} \end{bmatrix},
			    \end{align*}{}
			   where $\Asterisk$ contains unspecified elements. Using simple application of Linear Algebra we have that 
			   $$\rank\begin{bmatrix} A_N & B_N \end{bmatrix}=6N=n$$ 
			   and with use of Corollary \ref{thm:brockett_linear} we conclude that \eqref{eq:approx_FEM} satisfies Brockett's necessary condition for asymptotic stabilizability.
			\end{proof}
		\end{theorem}
	In a similar fashion we derive the following result for the ode system derived using MFEM.
	\begin{theorem}\label{theorem:FD_MFEM_stabilizablity}
			The approximated linear fully dynamic electromagnetic piezoelectric composite \eqref{eq:approx_MFEM} using MFEM, satisfies Brockett's necessary condition for asymptotic stabilizability through continuous state feedback. 
			\begin{proof}
			    Using \eqref{eq:approx_MFEM} we compute
			    \begin{align*}
				\begin{bmatrix}	A_N & B_N  \end{bmatrix}&\backsim \begin{bmatrix}
			    I_{6N} & \tilde{B}_{MFEM}
				\end{bmatrix},
				\end{align*}
			    with 
			    \begin{align*}
			         \tilde{B}^T_{MFEM} = \begin{bmatrix}\tilde{B}_1^T & \tilde{B}_2^T & \dots & \tilde{B}_N^T \end{bmatrix}
			    \end{align*}{}
			   where $\tilde{B}_i^T=\begin{bmatrix}0 & 0 & \Asterisk & 0 & 0 & 0\end{bmatrix}$ for $i\in\setof{1,\dots,N}$ with unspecified element $\Asterisk$. Using simple application of Linear Algebra we have that 
			   $$\rank\begin{bmatrix} A_N & B_N \end{bmatrix}=6N=n.$$ 
			  Using Corollary \ref{thm:brockett_linear}, we conclude that \eqref{eq:approx_MFEM} satisfies Brockett's necessary condition for asymptotic stabilizability.
			\end{proof}
		\end{theorem}
	
	 \begin{remark}{}
        For the approximated the piezoelectric actuator \eqref{eq:PDE_currentactuator_FD} we conjecture that the ode system using symbolic coefficients for both FEM and MFEM result in controllable systems. 
    \end{remark}

    In support of the established results in Theorem \ref{theorem:FD_FEM_stabilizablity} and Theorem \ref{theorem:FD_MFEM_stabilizablity}, we present {numerical} results showing the convergence of trajectories to zero {and the behaviour of the closed-loop eigenvalues} in the next section.

    

\section{{Numerical} results}
\blue{
In this section we present some numerical results, consisting of system trajectories and the behaviour of the eigenvalues of the approximated models. For second-order PDEs it has been shown that such systems are only stabilizable if it is stabilizable with collocated output feedback, see \cite{AmmariTucsnak2001} Therefore, we present the system trajectories of the closed-loop system $(A-BC)$, using $C:=B^T$. Moreover, we show the behaviour of the eigenvalues for large approximations order $N$, to show the behaviour of  the eigenvalues for the limit case $N\to\infty$. This gives an indication of what the stabilizability properties of the original PDE \eqref{eq:PDE_currentcomposite_FD} could be with the $-B^\ast$ feedback. \\
}

The simulations of the trajectories are executed using the approximation order $N=20$ and we use the same physical coefficients as in \cite{OzerTAC2019}, see Table \ref{tab:numerical_coefficients}. While it is simple to declare an initial state for one of the approximations with some sort of mechanical offset, it is rather challenging to find the corresponding initial state (especially the electromagnetic part) for the other approximation scheme. Therefore, we choose the initial state to be equal to a snapshot of the state-space of the open-loop systems of Fig \ref{fig:Validation_displacement} and use $t=845$ as a starting point for the closed-loop simulations.  The resulting longitudinal and traverse deflection of the closed-loop systems through electric output feedback  are presented in Fig \ref{fig:closedloop_FEM_deflections} and Fig \ref{fig:closedloop_MFEM_deflections} for the FEM and MFEM approximations, respectively. It can be seen in Fig \ref{fig:closedloop_FEM_deflections} and \ref{fig:closedloop_MFEM_deflections}, that both the longitudinal and transverse displacement converge to zero for both systems. The closed-loop simulations of the purely current actuated fully dynamical piezoelectric composite presented in \cite[Table II]{OzerTAC2019} using finite difference approximation scheme does not show convergence. \\

  \begin{figure}
		\includegraphics[width=\columnwidth]{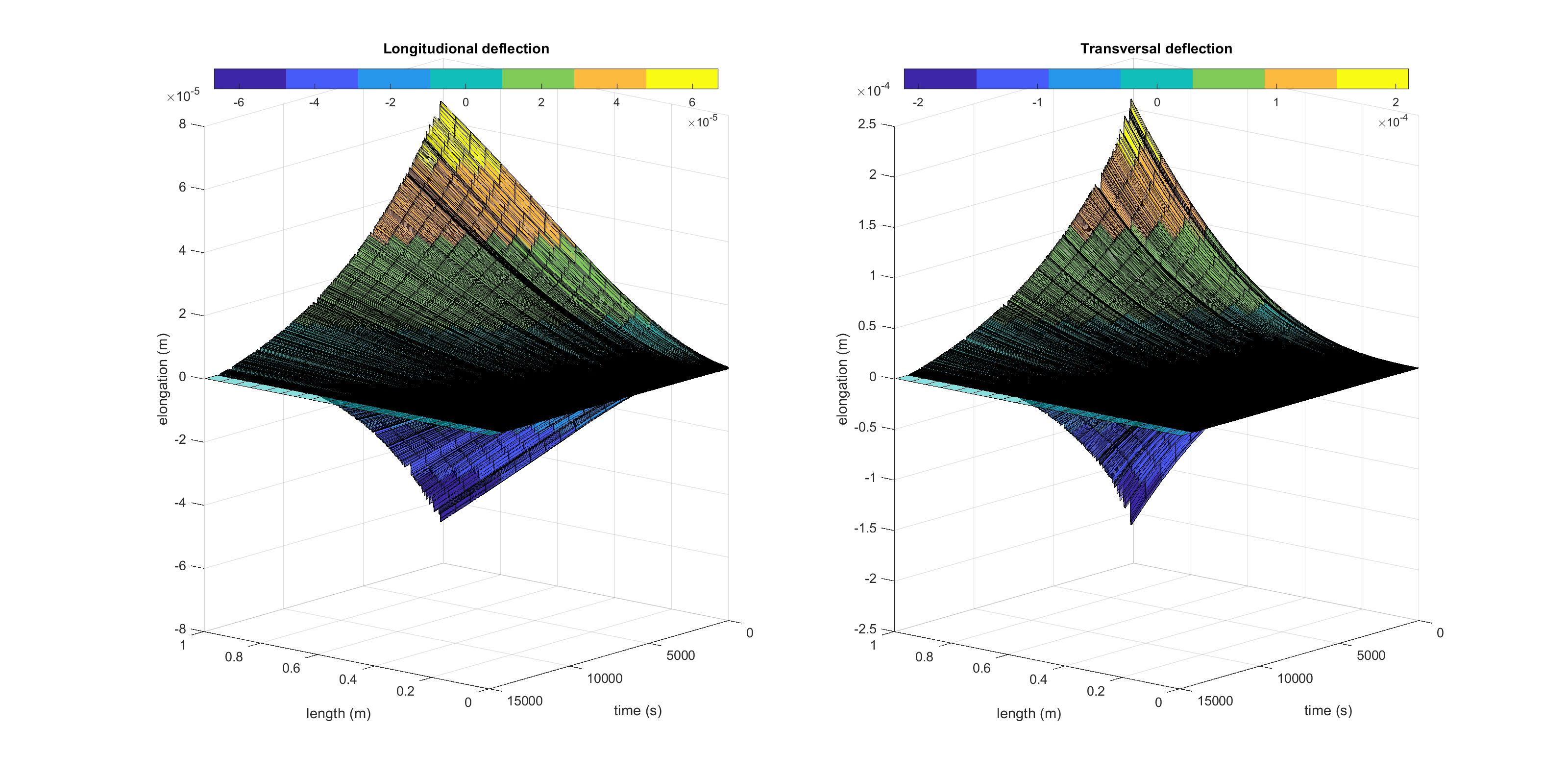}
		\caption{Closed-loop stabilizing trajectories. The longitudinal and traverse deflection of the fully dynamic piezoelectric composite using FEM approximations} 
		\label{fig:closedloop_FEM_deflections}
	\end{figure}
	
	  \begin{figure}
		\includegraphics[width=\columnwidth]{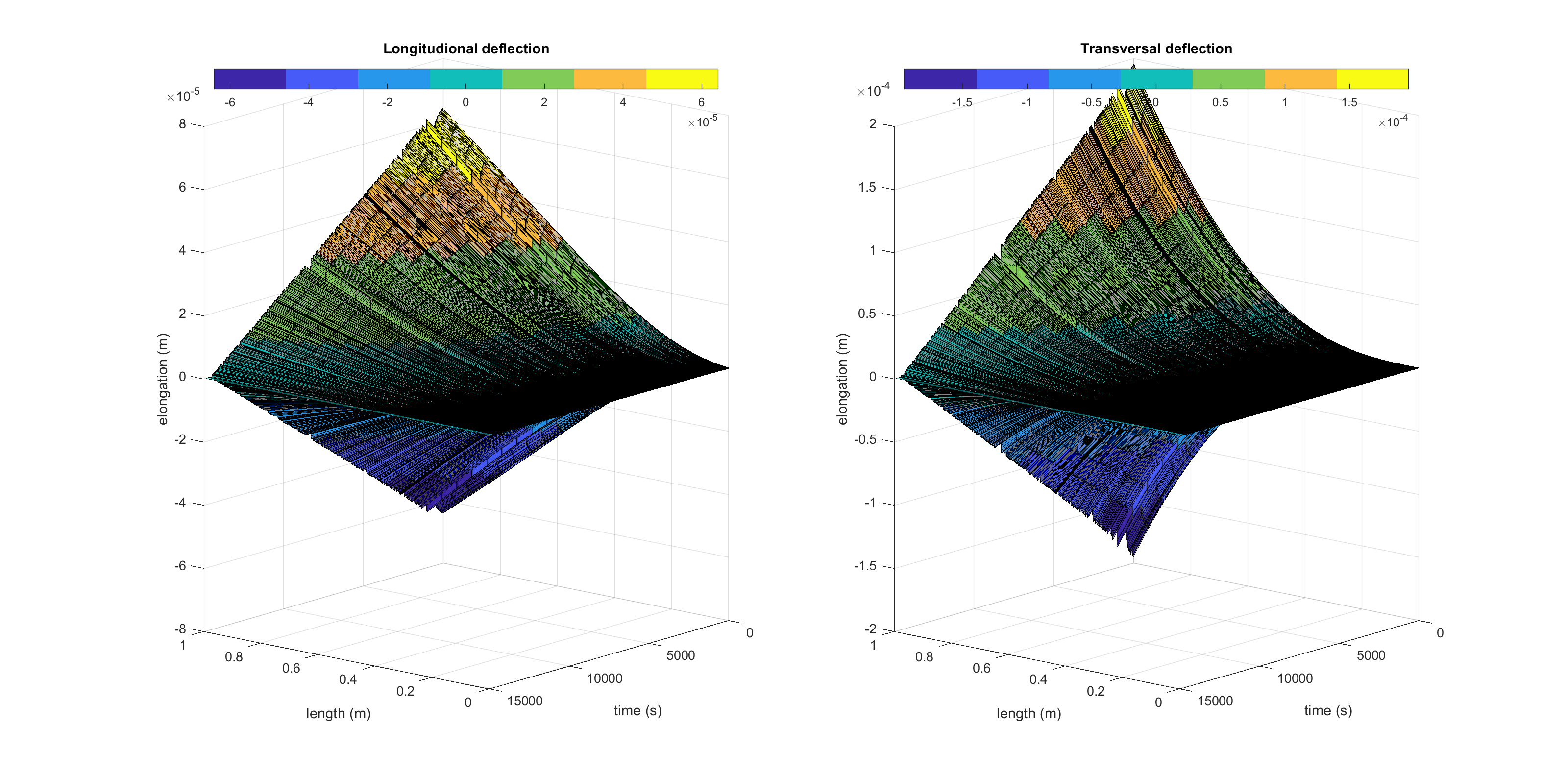}
		\caption{Closed-loop stabilizing trajectories. The longitudinal and traverse deflection of the fully dynamic piezoelectric composite using MFEM approximations.} 
		\label{fig:closedloop_MFEM_deflections}
	\end{figure}

    \begin{figure}
		\includegraphics[width=\columnwidth]{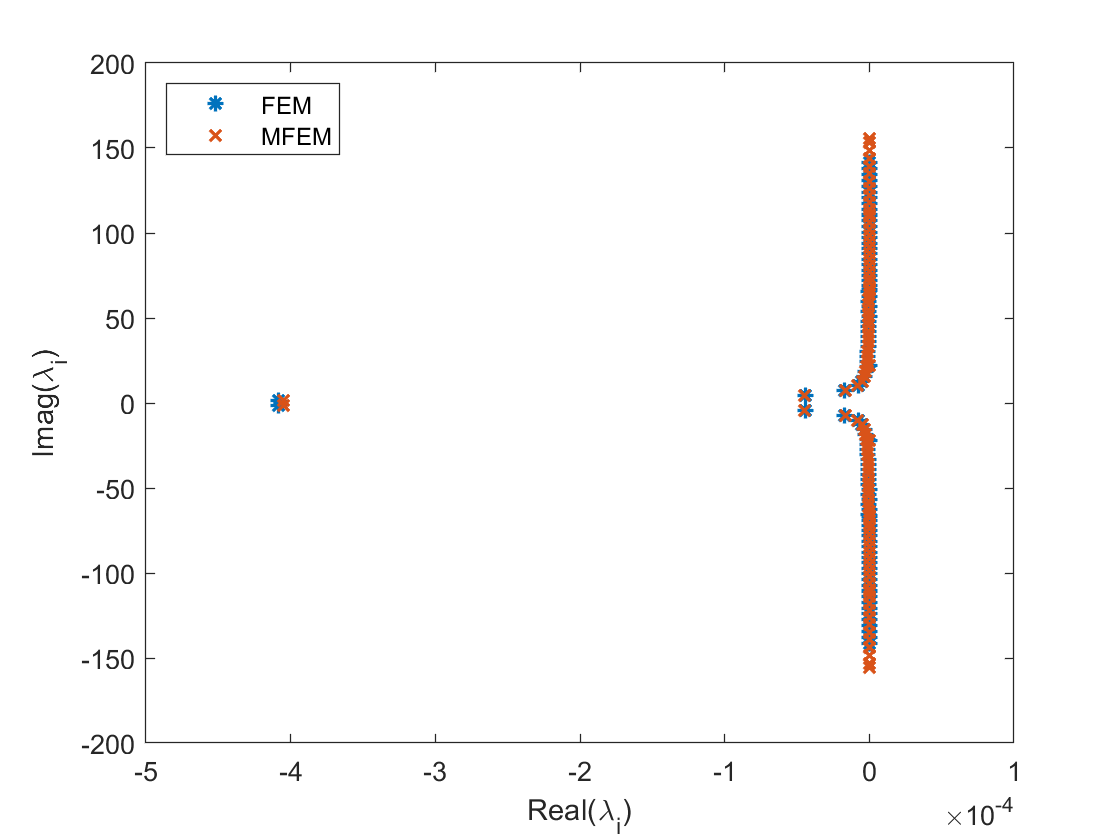}
		\caption{Eigenvalues of the closed-loop systems.} 
		\label{fig:closedloop_asymptotic_eigenvalues}
	\end{figure}

    Next to the stabilizing trajectories depicted in Fig \ref{fig:closedloop_FEM_deflections} and Fig \ref{fig:closedloop_MFEM_deflections}, we present the behaviour of the eigenvalues for the large approximation order $N=100$. In Fig \ref{fig:closedloop_asymptotic_eigenvalues}, we show the first 100 eigenvalues $\lambda_i$ for $i\in\setof{0,\dots,100}$ with imaginary part closes to zero. As can be seen in Fig \ref{fig:closedloop_asymptotic_eigenvalues}, the eigenvalues for both approximation methods show asymptotic behaviour towards the imaginary axis for large frequencies, i.e. $\im(\lambda_i)>>0$. This phenomenon indicates asymptotic stability of the closed-loop system using collocated output-feedback for the limit case $N\to\infty$ and perhaps also the original PDE \cite{CurtainZwart1995introduction,Luo1999SSIDSA}.

\section{Conclusion}
In this work, we have proposed a novel purely current actuated piezoelectric composite model with a fully dynamic electromagnetic field assumption. The novelty of this model results from the definition of the magnetic-flux, inspired by the procedure of deriving voltage actuated models, which makes it possible to circumvent the use of magnetic vector potentials and a gauge function to produce a purely current actuated piezoelectric composite model. The existing modelling framework of piezoelectric beams and composites, resulting in models actuated by either voltage, charge, or current using magnetic vector potentials, is now extended with current actuated models employing the magnetic-flux. Furthermore, we have shown that the approximations of this novel current actuated model satisfies a necessary condition for stabilizability using either a finite element method or mixed finite element method. Moreover,  the closed-loop trajectories and eigenvalue behaviour shown in the numerical results, support the idea of an asymptotically stabilizing piezoelectric composite through electric output feedback. \\

{\em Future Work:}  It is interesting to investigate the asymptotic stabilizability properties of the original PDE and see if these align with the analytic and numerical stabilizability results of the approximations. Furthermore, it would be interesting to see the stabilizability of the novel piezoelectric composite using different boundary conditions (e.g. fixed-fixed).

\bibliographystyle{IEEEtran}

\end{document}